\documentclass[11pt,reqno]{amsart}
\usepackage{amsmath,amsfonts,amsthm,amssymb,amscd}
\usepackage{float,graphicx}
\usepackage[letterpaper,top=3cm,bottom=3cm,left=3cm,right=3cm,marginparwidth=1.75cm]{geometry}

\usepackage{color,xcolor}
\usepackage{ulem}
\usepackage{cancel}

\newcommand{\cN}{\mathcal{N}}

\newcommand{\bR}{\mathbb{R}}

\newtheorem{theorem}{Theorem}

\newtheorem{definition}{Definition}

\newtheorem{proposition}{Proposition}

\newtheorem{assumption}{Assumption}
\newtheorem{remark}{Remark}

\begin{document}
    
    \title[Optimal transport natural gradient]{Optimal transport natural gradient for statistical manifolds with continuous sample space}
    \author[Chen]{Yifan Chen*}
    \address{Department of Computing + Mathematical Sciences \\
        California Institute of Technology \\
        Pasadena, California 91106 \\
        email: yifanc@caltech.edu \\
        Corresponding author}
    
    \author[Li]{Wuchen Li}
    \address{Department of Mathematics \\
        UCLA \\
        Los Angeles, CA 90095 USA \\
        email: wcli@math.ucla.edu}
    
    
    \maketitle
    \begin{abstract}
        We study the Wasserstein natural gradient in parametric statistical models with continuous sample spaces. Our approach is to pull back the $L^2$-Wasserstein metric tensor in the probability density space to a parameter space, equipping the latter with a positive definite metric tensor, under which it becomes a Riemannian manifold, named the Wasserstein statistical manifold. In general, it is not a totally geodesic sub-manifold of the density space, and therefore its geodesics will differ from the Wasserstein geodesics, except for the well-known Gaussian distribution case, a fact which can also be validated under our framework. We use the sub-manifold geometry to derive a gradient flow and natural gradient descent method in the parameter space. When parametrized densities lie in $\bR$, the induced metric tensor establishes an explicit formula. In optimization problems, we observe that the natural gradient descent outperforms the standard gradient descent when the Wasserstein distance is the objective function. In such a case, we prove that the resulting algorithm behaves similarly to the Newton method in the asymptotic regime. The proof calculates the exact Hessian formula for the Wasserstein distance, which further motivates another preconditioner for the optimization process. To the end, we present examples to illustrate the effectiveness of the natural gradient in several parametric statistical models, including the Gaussian measure, Gaussian mixture, Gamma distribution, and Laplace distribution.
    \end{abstract}
    
    \section{Introduction}
    The statistical distance between probability measures plays an important role in many fields such as data analysis and machine learning, which usually consist in minimizing a loss function as
    \begin{equation*}
    \text{minimize}\quad  d(\rho,\rho_{e})\quad \text{s.t.}\quad \rho \in \mathcal{P}_{\theta}.
    \end{equation*}
    Here $\mathcal{P}_{\theta}$ is a parameterized subset of the probability density space, and $\rho_e$ is the target density, which is often an empirical realization of a ground-truth distribution. The function $d$ quantifies the difference between densities $\rho$ and $\rho_e$.
    
    An important example for $d$ is the Kullback-Leibler (KL) divergence, also known as the relative entropy, which closely relates to the maximum likelihood estimate in statistics and the field of information geometry \cite{IG}\cite{IG2}. The Hessian operator of KL embeds $\mathcal{P}_{\theta}$ as a statistical manifold, in which the Riemannian metric is the Fisher-Rao metric. Due to Chentsov \cite{cencov}, the Fisher-Rao metric is the only one, up to scaling, that is invariant under sufficient statistics. Using the Fisher-Rao metric, a natural gradient descent method, realized by a Forward-Euler discretization of the gradient flow in the manifold, has been introduced. It has found many successful applications in a variety of problems such as blind source separation \cite{Amari1998Adaptive}, machine learning \cite{NG}\cite{ Martens2014New}, filtering \cite{Ollivier NG Kalman}\cite{Ollivier NG extended}, statistics \cite{Malag NG exponential family}\cite{Malag NG fitness}, optimization \cite{Malag RS for NG}\cite{Ollivier Asym NG}\cite{Olliver IGO} and many others.
    
    Recently, the Wasserstein distance, introduced through the field of optimal transport (OT), has been attracting increasing attention in computation and learning \cite{PeyreCuturi2018_computational}. One promising property of the Wasserstein distance is its ability to reflect the metric on sample space, rendering it very useful in machine learning \cite{WGAN}\cite{LWL}\cite{Boltzman}, statistical models \cite{CarliNingGeorgiou2013_convexa}\cite{ChenGeorgiouTannenbaum2017_optimal} and geophysics \cite{ Engquist2014Application}\cite{Engquist2016Optimal}\cite{ Chen2017The}\cite{MeBrMeOuVi:16}\cite{MeBrMeOuVi:16b}. Further, OT theory provides the $L^2$-Wasserstein metric tensor, which gives the probability density space (with smooth, positive densities in a compact domain) a formal infinite-dimensional Riemannian differential structure \cite{Lafferty}\cite{Lott}. This structure can be made mathematically rigorous and general to work for probability measures with finite second moments by using tools in metric geometry \cite{Ambrosio2008Gradient}. Under such a setting, the gradient flow with respect to the $L^2$-Wasserstein metric tensor, known as the Wasserstein gradient flow, is well-defined and has been seen deep connections to fluid dynamics \cite{C2}\cite{otto2001}, differential geometry \cite{LV} and mean-field games \cite{Ligame}\cite{MFG}. 
    
    Nevertheless, compared to the Fisher-Rao metric, the Riemannian structure of the Wasserstein metric is mostly investigated in the whole probability space rather than the parameterized subset $\mathcal{P}_{\theta}$. Therefore, there remains a gap in developing the natural gradient concept in a parametric model within the Wasserstein geometry context. Here we are primarily interested in the question of whether there exists the Wasserstein metric tensor and the associated Wasserstein natural gradient in a general parameterized subset and whether we can gain computational benefits by considering these structures. We believe the answer to it will serve as a window to bring synergies between information geometry and optimal transport communities. {We note that our motivation is very similar to \cite{LM}, which introduces the natural gradient concept into parametric statistical models with discrete sample spaces, and discusses its connection and difference to the Fisher-Rao natural gradient in information geometry.}
    
    In this paper, we embed the Wasserstein geometry to parametric probability models with continuous sample spaces. {Our treatment relies on the same ``pull-back'' idea as in \cite{LM}. Precisely,} we pull back the $L^2$-Wasserstein metric tensor into the parameter space, making it become a finite-dimensional Riemannian manifold, given the resulted tensor is positive definite. In \cite{LM}, where the finite discrete sample space and positive probability simplex are considered, the pull-back relation implicitly defines the metric tensor in the parameter space, and its explicit form involves a weighted Laplacian on the graph. In our continuous world, we need solutions to an elliptic equation to define the associated metric tensor. We rely on Assumption \ref{assume: elliptic} to guarantee the invertibility of this operator, which makes sure the metric tensor is well-defined. This assumption holds for the compact sample space and positive smooth densities, and also holds for Gaussian distributions in the unbounded $\bR^d$, as we will see in section \ref{eg: Gaussian}. 
    
    We remark that, in general, the parametric model will not be a totally geodesic submanifold of the density manifold, except for the well-known Gaussian Wasserstein case \cite{WM}\cite{GW}\cite{BuresWasserstein}, see also section \ref{eg: Gaussian}. Thus generally, we will obtain a new geodesic distance in the parametric model that is different from the Wasserstein metric. However, this will not be a problem when it is used for specific computational purposes, as we will demonstrate both theoretically (in Theorem \ref{thm: hessian of R} and Proposition \ref{prop: hessian and G}) and numerically (in section \ref{section:examples}), the natural gradient concept derived from the submanifold geometry will still approximate the second-order information of the geodesic distance in the whole density manifold. In the meantime, we will benefit from the fact that the new geometry is on the finite-dimensional parameter space. This makes it easier to deal with mathematically compared to the infinite-dimensional density manifold.
    
    The Riemannian geometry in the parameter space allows us to derive the constrained Wasserstein gradient flow in it. The discretized version of the flow leads to the Wasserstein natural gradient descent method, in which the induced metric tensor acts as a preconditioning term in the standard gradient descent iteration. This is a standard approach for introducing the natural gradient concept and has been used in the Fisher-Rao natural gradient \cite{NG} and the Wasserstein natural gradient (discrete sample space) \cite{LM}. In our Wasserstein natural gradient (continuous sample space), when the dimension of densities is one, we obtain an explicit formula of the metric tensor. Precisely, given $\rho(x,\theta)$ as a parameterized density, $x\in \mathbb{R}^1$ and $\theta\in \Theta\subset \mathbb{R}^d$, the $L^2$-Wasserstein metric tensor on $\Theta$ will be
    \begin{align*}
    G_W(\theta) = \int \frac{1}{\rho(x,\theta)}(\nabla_{\theta} F(x,\theta))^T\nabla_{\theta} F(x,\theta)dx,
    \end{align*}    
    where $F(y,\theta)=\int_{-\infty}^x\rho(y,\theta)dy$ is the cumulative distribution function of $\rho(x,\theta)$. We apply the natural gradient descent induced by $G_W(\theta)$ to the Wasserstein metric modeled problems. It is seen that the Wasserstein gradient descent outperforms the Euclidean and Fisher-Rao natural gradient descent in the iterations. We give theoretical justifications of this phenomenon by showing that the Wasserstein gradient descent behaves asymptotically as the Newton method in such a case. A detailed description of the Hessian matrix is also presented by leveraging techniques in one-dimensional OT. Interestingly, this formula also provides us with a new preconditioner for the Wasserstein metric modeled problems and results in a new algorithm which we call the modified Wasserstein gradient descent. We will compare its performance with the Wasserstein gradient descent in the experiments.
    
    In the literature, there are pioneers toward the constrained Wasserstein gradient flow. \cite{C2} studies the density space with a fixed mean and variance. Compared to them, we focus on a density set parameterized by a finite-dimensional parameter space. Also, there have been many works linking information geometry and optimal transport \cite{LP}\cite{Wong}. In particular, the Wasserstein metric tensor for Gaussian distributions exhibits an explicit form {\cite{WM}\cite{GW}\cite{BuresWasserstein}}, which leads to extensive studies between the Wasserstein and Fisher-Rao metric for this model \cite{Marti2016Optimal}\cite{IGW}\cite{Sanctis2017A}. In contrast to their works, our Wasserstein metric tensor can work for general parametric models. For consistency, in section \ref{eg: Gaussian}, we will show that our defined Wasserstein statistical manifold gives the same metric tensor as in the literature when the parametric model is Gaussian. Thus, it can be seen as a direct extension to the Gaussian case. Under such an extension, we are able to discuss the natural gradient for a lot of parametric models in a systematic way.
    
    This paper is organized as follows. In section \ref{section:review of OT}, we briefly review the theory of optimal transport, with a concentration on its Riemannian differential structure. In section \ref{section:wasserstein ng}, we introduce the Wasserstein statistical manifolds by defining the metric tensor in the parameter space directly through the pull-back relation. The Wasserstein gradient flow and natural gradient descent method are then derived. We give a concise study of the metric tensor for one-dimensional densities, showing its connection to the Fisher information matrix. In this case, we theoretically analyze the effect of this natural gradient in the Wasserstein metric modeled problems. In section \ref{section:examples}, examples are presented to justify the previous discussions. {We also provide a detailed comparison between the performance of the Fisher-Rao and Wasserstein natural gradient in different inference tasks in section \ref{sec: comparison FR and W}, with regard to different choices of loss functions, metric tensors and whether or not the ground truth density lies in the parametric family. Finally, we conclude the paper with several discussions in section \ref{sec: discussion}.}      
    \section{Review of Optimal Transport Theory}\label{review}
    \label{section:review of OT}
    In this section, we briefly review the theory of OT. We note that there are several equivalent definitions of OT, ranging from static to dynamic formulations. In this paper, we focus on the dynamic formulation and its induced Riemannian metric tensor in the density space.     
    
    The optimal transport problem is firstly proposed by Monge in 1781: given two probability densities $\rho^0, \rho^1$ on $\Omega \subset \bR^n$ (in general situation, probability measures are considered \cite{Ambrosio2008Gradient}\cite{vil2008}; here for simplicity, our focus is on probability densities), the goal is to find a transport plan $T: \Omega \to \Omega$ pushing $\rho^0$ to $\rho^1$ that minimizes the whole transportation cost, i.e.
    \begin{equation}
    \label{eqn:Monge}
    \inf_{T} \int_{\Omega} d\left(x,T(x)\right) \rho^0(x) dx\quad \text{s.t.} \quad \int_{A} \rho^1(x)dx=\int_{T^{-1}(A)} \rho^0(x)dx,
    \end{equation}
    for any Borel subset $A \subset \Omega$. Here the function $d\colon \Omega\times \Omega\rightarrow \mathbb{R}$ is the ground cost that measures the efforts to pay for transporting $x$ to $T(x)$. In the whole discussions we set $d(x,y)=\|x-y\|^2$ as the square of Euclidean distance. We assume all the densities belong to $\mathcal{P}_2(\Omega)$, the collection of probability density functions on $\Omega \subset \mathbb{R}^n$ with finite second moments. For simplicity, we also assume the densities under consideration are smooth in this paper, such that we are not confronted with delicate differentiability issues. 
    
    In 1942, Kantorovich relaxed the problem to a linear programming:
    \begin{equation}
    \label{eq:Kantorovich}
    \min_{\pi\in \Pi(\rho^0, \rho^1)}\int_{\Omega\times\Omega}\|x-y\|^2\pi(x,y)dxdy,
    \end{equation}
    where the infimum is taken {over} {the set} $\Pi$ {of joint probability measures} {on} $\Omega\times\Omega$ that have marginals $\rho^0$, $\rho^1$. This formulation finds a wide array of applications in computations \cite{PeyreCuturi2018_computational}.
    
    In recent years, OT connects to a variational problem in density space, known as the Benamou-Brenier formula \cite{BB}:
    \begin{subequations}\label{BB1}
        \begin{equation}\label{BB}
        \inf_{\Phi_t}~\int_0^1\int_\Omega \|\nabla\Phi(t,x)\|^2\rho(t,x) dx dt, 
        \end{equation}
        where the infimum is taken over the set of Borel potential functions {$\Phi: [0,1]\times \Omega\to \mathbb{R}$.} 
        Each gradient vector field of potential $\Phi_t=\Phi(t,x)$ on sample space determines a corresponding density path $\rho_t=\rho(t,x)$ as the solution of the continuity equation:  
        \begin{equation}\label{BB2}
        \frac{\partial \rho(t,x)}{\partial t}+\nabla\cdot (\rho(t,x)\nabla\Phi(t,x))=0,\quad \rho(0,x)=\rho^0(x),\quad \rho(1,x)=\rho^1(x).
        \end{equation}
    \end{subequations}
    Here $\nabla\cdot$ and $\nabla$ are the divergence and gradient operators in $\mathbb{R}^n$. If $\Omega$ is a compact set, the zero flux condition (Neumann condition) is proposed on the boundary of $\Omega$. This is to ensure that $\int_\Omega\frac{\partial\rho(t,x)}{\partial t}dx=0$, so that the total mass is conserved. 
    
    Under mild regularity assumptions, the above three formulations \eqref{eqn:Monge} \eqref{eq:Kantorovich} \eqref{BB1}  
    are equivalent \cite{vil2008}. A convenient example of the assumption is that the probability measures are absolutely continuous with respect to the Lebesgue measure \cite{vil2008}. Since in our set-up, $\rho^0, \rho^1$ are probability densities, the assumption holds and all the three optimization problems share the same optimal value.  The value is denoted by $(W_2(\rho^0,\rho^1))^2$, which is called the square of the $L^2$-Wasserstein distance between $\rho^0$ and $\rho^1$.
    Here the subscript ``2'' in $W_2$ indicates that the $L^2$ ground distance is used. We note that formulation  \eqref{eqn:Monge} \eqref{eq:Kantorovich} is static, in the sense that only the initial and final states of the transportation are considered. By taking the transportation path into consideration, OT enjoys a dynamical formulation \eqref{BB1}. This will be our primary interest in the following discussion.
    
    The variational formulation~\eqref{BB1} introduces a formal infinite dimensional Riemannian manifold in the density space. 
    We note that the infinite-dimensional Riemannian geometry is well-defined for smooth positive densities in a compact domain \cite{Lott}, and can be extended to more general situations where the measure is only with finite second moments through tools in metric geometry \cite{Ambrosio2008Gradient}. Here, for better illustration, assume $\Omega$ is compact and consider {the set of smooth and strictly positive densities}
    \begin{equation*}
    \mathcal{P}_+(\Omega)=\Big\{\rho \in C^{\infty}(\Omega)\colon \rho(x)>0,~\int_\Omega\rho(x)dx=1\Big\}\subset\mathcal{P}_2(\Omega).  \end{equation*}
    Denote by $\mathcal{F}(\Omega):=C^{\infty}(\Omega)$ the set of smooth real valued functions on $\Omega$. 
    The tangent space of $\mathcal{P}_+(\Omega)$ is given by 
    $$
    T_\rho\mathcal{P}_+(\Omega) = \Big\{\sigma\in \mathcal{F}(\Omega)\colon \int_\Omega\sigma(x) dx=0 \Big\}.
    $$
    We also denote $\tilde{\mathcal{F}}(\Omega):=\{\sigma \in \mathcal{F}(\Omega)\colon \frac{\partial \sigma}{\partial n}|_{\partial \Omega}=0\}$.
    Given $\Phi\in \tilde{\mathcal{F}}(\Omega) / \bR$ and $\rho\in \mathcal{P}_+(\Omega)$, define
    \begin{equation*}
    V_{\Phi}(x):=-\nabla\cdot (\rho(x) \nabla \Phi(x))\in T_{\rho}\mathcal{P}_+(\Omega).
    \end{equation*}
    Since $\rho$ is positive in a compact region $\Omega$, the elliptic operator identifies the function $\Phi \in \tilde{\mathcal{F}}(\Omega)/\bR$ with the tangent vector $V_{\Phi}$ in $\mathcal{P}_+(\Omega)$. This gives an isomorphism
    $$\tilde{\mathcal{F}}(\Omega)/\bR\rightarrow T_{\rho}\mathcal{P}_+(\Omega), \quad \Phi \mapsto V_\Phi.$$
    
    So we can treat  $T^*_{\rho}\mathcal{P}_+(\Omega)\cong\tilde{\mathcal{F}}(\Omega)/\bR$ as the smooth cotangent space of $\mathcal{P}_+(\Omega)$. The above facts imply that the $L^2$-Wasserstein metric tensor on the density space \cite{Lott}\cite{LiG} can be obtained as follows:
    \begin{definition}[$L^2$-Wasserstein metric tensor]\label{eqn:riemannian whole density space}
        Define the inner product on the tangent space of positive densities $g_\rho\colon {T_\rho}\mathcal{P}_+(\Omega)\times {T_\rho}\mathcal{P}_+(\Omega)\rightarrow \mathbb{R}$ by
        \begin{equation*}
        g_\rho(\sigma_1, \sigma_2)=\int_\Omega \nabla\Phi_1(x)\cdot\nabla\Phi_2(x)\rho(x) dx,
        \end{equation*}
        where $\sigma_1=V_{\Phi_1}$, $\sigma_2=V_{\Phi_2}$ with $\Phi_1(x)$, $\Phi_2(x)\in \tilde{\mathcal{F}}(\Omega)/\mathbb{R}$. 
    \end{definition}
    With the inner product specified above,  the variational problem~\eqref{BB1} becomes a geometric action energy in $(\mathcal{P}_+(\Omega), g_\rho)$. As in Riemannian geometry, the square of distance equals the energy of geodesics, i.e.
    \begin{equation*}
    (W_2(\rho^0, \rho^1))^2=\inf_{\Phi_t}~\Big\{\int_0^1 g_{\rho_t}(V_{\Phi_t}, V_{\Phi_t}) dt\colon \partial_t\rho_t=V_{\Phi_t},~ \rho(0,x)=\rho^0,~\rho(1,x)=\rho^1\Big\}.
    \end{equation*}
    This is exactly the form in \eqref{BB1}. In this sense, it explains that the dynamical formulation of OT exhibits the Riemannian manifold structure in the density space. In \cite{Lafferty}, $(\mathcal{P}_+(\Omega), g_\rho)$ is named density manifold. More geometric studies are provided in \cite{Lott}\cite{LiG}.
    
    \section{Wasserstein Natural Gradient}
    \label{section:wasserstein ng}
    In this section, we study parametric statistical models, which are parameterized subsets of the probability space $\mathcal{P}_2(\Omega)$. We pull back the $L^2$-Wasserstein metric tensor into the parameter space, turning it to be a Riemannian manifold. This consideration allows us to introduce the Riemannian (natural) gradient flow on parameter spaces, which further leads to a natural gradient descent method in the field of optimization. When densities lie in $\bR$, we show that the metric tensor establishes an explicit formula. It acts as a positive definite and asymptotically-Hessian preconditioner for the Wasserstein metric related minimizations. 
    
    \subsection{Wasserstein statistical manifold} We adopt the definition of a statistical model from \cite{IG2} (Chapter 3.2). It is represented by a triple $(\Omega,\Theta,\rho)$, where $\Omega\subset\mathbb{R}^n$ is the continuous sample space, $\Theta \subset \bR^d$ is the statistical parameter space, and $\rho$ is the probability density on $\Omega$ parameterized by $\theta$ such that $\rho\colon \Theta\rightarrow \mathcal{P}_2(\Omega)$ and $\rho=\rho(\cdot,\theta)$. In Chapter 3.2 of the book \cite{IG2} some differentiability conditions are posed on the map $\rho$ between Banach manifolds $\Theta$ and $\mathcal{P}_2(\Omega)$. Here for simplicity we assume $\Omega$ is either compact or $\Omega=\bR^n$, and each $\rho(x,\theta)$ is positive and smooth. The parameter space is a finite dimensional manifold with metric tensor denoted by $\langle\cdot,\cdot\rangle_{\theta}$. {The Riemannian gradient of a function $\rho(\theta)$ on $\Theta$ is denoted by $\nabla_{\theta}\rho(\theta)$.} We also use $\langle\cdot,\cdot \rangle$ to represent the Euclidean inner product in $\mathbb{R}^d$. 
    
    The Riemannian metric $g_{\theta}$ on $\Theta$ will be the {formal} pull-back of $g_{\rho(\cdot,\theta)}$ on $\mathcal{P}_2(\Omega)$. That is, for $\xi$, $\eta\in T_\theta\Theta$, we have
    \begin{equation}
    \label{eqn:pull-back}
    g_\theta(\xi,\eta):=g_{\rho(\cdot,\theta)}(d_\xi \rho(\theta), d_\eta\rho(\theta)),
    \end{equation}
    where {by definition of the Riemannian gradient},  $d_\xi \rho(\theta)=\langle\nabla_\theta\rho(\cdot,\theta),\xi\rangle_{\theta}$, $d_\eta\rho(\theta)=\langle\nabla_\theta\rho(\cdot,\theta),\eta\rangle_{\theta}$, in which we use the notation that $d_{\xi}\rho(\theta)$ is the derivative of $\theta \to \rho(\theta)$ in the direction of $\xi$ computed at $\theta$. The relation \eqref{eqn:pull-back} implicitly defines the tensor $g_{\rho(\cdot,\theta)}$. It involves the solution of elliptic equations and in order to make the formula explicit we make the following assumptions on the statistical model $(\Omega,\Theta,\rho)$:
    \begin{assumption}
        For the statistical model $(\Omega,\Theta,\rho)$, one of the following two conditions are satisfied:
        \begin{enumerate}
            \item The sample space $\Omega$ is compact, and for each $\xi \in T_{\theta}(\Theta)$, the elliptic equation
            \begin{equation*}
            \begin{cases}
            -\nabla \cdot (\rho(x,\theta) \nabla \Phi(x))=\langle \nabla_{\theta} \rho(x,\theta), \xi \rangle_{\theta}\\
            \frac{\partial \Phi}{\partial n}|_{\partial \Omega}=0
            \end{cases}
            \end{equation*}
            has a smooth solution $\Phi$ satisfying 
            \begin{equation}
            \int_\Omega \rho(x,\theta) \|\nabla \Phi(x)\|^2 dx < +\infty.
            \label{eqn: regularity of phi}
            \end{equation}    
            \item The sample space $\Omega=\bR^n$, and for each $\xi \in T_{\theta}(\Theta)$, the elliptic equation
            \begin{equation*}
            -\nabla \cdot (\rho(x,\theta) \nabla \Phi(x))=\langle \nabla_{\theta} \rho(x,\theta), \xi \rangle_{\theta}
            \end{equation*}
            has a smooth solution $\Phi$ satisfying 
            \begin{equation}
            \int_\Omega \rho(x,\theta) \|\nabla \Phi(x)\|^2 dx < +\infty \quad \text{and} \quad \int_\Omega \rho(x,\theta) |\Phi(x)|^2 dx < +\infty.
            \label{eqn: regularity of phi, whole space}
            \end{equation}    
        \end{enumerate}
        \label{assume: elliptic}
    \end{assumption}
    {Note that $\Phi$ will be dependent on $\xi$, the tangent vector at play. For simplicity of notation we do not write this dependence explicitly, but this will be assumed wherever it appears throughout this paper.}
    Assumption \ref{assume: elliptic} guarantees the action of ``pull-back'' we describe above is well-defined. {The condition \eqref{eqn: regularity of phi} is for the boundness of the energy of the solution.} The {additional assumption in} condition \eqref{eqn: regularity of phi, whole space} {is} used to {guarantee the vanishing of boundary terms when integrating by parts in the non-bounded domain, i.e., one has $-\int_{\bR^n} \Phi_1\nabla \cdot \rho \nabla \Phi_2=\int_{\bR^n} \rho \nabla \Phi_1 \cdot \nabla \Phi_2$ for any $\Phi_1,\Phi_2$ satisfying condition \eqref{eqn: regularity of phi, whole space}.  This will be used in the proof of the uniqueness of the solution to the equation. We will use the uniqueness result  in section \ref{eg: Gaussian} for Gaussian measures, to show the submanifold geometry is totally geodesic.} 
    \begin{proposition}
        Under assumption \ref{assume: elliptic},  the solution $\Phi$ is unique modulo the addition of a spatially-constant function.
    \end{proposition}
    
    \begin{proof}
        It suffices to show the equation
        \begin{equation}
        \nabla \cdot (\rho(x,\theta) \nabla \Phi(x))=0
        \label{homogeneous elliptic}
        \end{equation}
        only has the trivial solution $\nabla\Phi=0$ in the space described in assumption \ref{assume: elliptic}.
        
        For case (1), we multiple $\Phi$ to \eqref{homogeneous elliptic} and integrate it in $\Omega$. Integration by parts result in
        \[\int_\Omega \rho(x,\theta)\|\nabla \Phi(x)\|^2=0 \]
        due to the zero flux condition. Hence $\nabla \Phi=0$.
        
        For case (2), we denote by $B_R(0)$ the ball in $\bR^n$ with center $0$ and radius $R$. Multiply $\Phi$ to the equation and integrate in $B_R(0)$:
        \[\int_{B_R(0)} \rho(x,\theta)\|\nabla \Phi(x)\|^2=\int_{\partial B_R(0)}\rho(x,\theta)\Phi(x)(\nabla \Phi(x) \cdot n)dx .\]
        By Cauchy-Schwarz inequality we can control the right hand side by
        \begin{equation*}
        |\int_{\partial B_R(0)}\rho(x,\theta)\Phi(x)(\nabla \Phi(x) \cdot n) dx |^2 \leq \int_{\partial B_R(0)}\rho(x,\theta)\Phi(x)^2 dx \cdot \int_{\partial B_R(0)}\rho(x,\theta)\|\nabla \Phi(x)\|^2 dx.
        \end{equation*}
        However, due to \eqref{eqn: regularity of phi, whole space}, there exists a sequence $R_k, k\geq 1$, such that $R_{k+1}>R_k$, $\lim_{k\rightarrow +\infty} R_k =\infty$ and 
        \[\lim_{k\rightarrow+\infty} \int_{\partial B_{R_k}(0)}\rho(x,\theta)\Phi(x)^2 dx = \lim_{k\rightarrow+\infty} \int_{\partial B_{R_k}(0)}\rho(x,\theta)\|\nabla \Phi(x)\|^2 dx = 0.\]
        Hence 
        \[\lim_{k\rightarrow+\infty} |\int_{\partial B_{R_k}(0)}\rho(x,\theta)\Phi(x)(\nabla \Phi(x) \cdot n) dx |^2 = 0,\]
        which leads to 
        \[\int_{\bR^n} \rho(x,\theta)\|\nabla \Phi(x)\|^2=0. \]
        Thus $\nabla \Phi=0$, which is the trivial solution.
    \end{proof}
    
    Since we deal with positive $\rho$, the existence of solutions to the case (1) in Assumption \ref{assume: elliptic} is ensured by the theory of elliptic equations. For case (2), i.e., $\Omega=\bR^n$, we show when $\rho$ is Gaussian distribution in $\bR^d$, the existence of solution $\Phi$ is guaranteed and exhibits explicit formulation in section \ref{eg: Gaussian}. Although we only deal with compact $\Omega$ or the whole $\bR^n$, the treatment to some other $\Omega$, such as the half-space of $\bR^n$, is similar and omitted. {Given these preparations, now we can explicitly write down the Wasserstein metric tensor in the parameter space. See the following Definition \ref{def:metric tensor} and Proposition \ref{prop: metric tensor on theta}.}
    
    \begin{definition}[$L^2$-Wasserstein metric tensor in parameter space]
        \label{def:metric tensor} Under assumption \ref{assume: elliptic}, the inner product $g_{\theta}$ on $T_{\theta}(\Theta)$ is defined as
        \begin{equation*}
        g_{\theta}(\xi, \eta)=\int_\Omega \rho(x, \theta) \nabla \Phi_\xi(x) \cdot \nabla \Phi_\eta(x) dx,
        \end{equation*}
        where $\xi,\eta$ are tangent vectors in $T_{\theta}(\Theta)$, $\Phi_\xi$ and $\Phi_\eta$ satisfy $\langle \nabla_{\theta} \rho(x,\theta), \xi \rangle_{\theta} = -\nabla \cdot (\rho \nabla \Phi_{\xi}(x))$ and $\langle \nabla_{\theta} \rho(x,\theta), \eta \rangle_{\theta} =-\nabla \cdot (\rho \nabla \Phi_{\eta}(x))$.
    \end{definition}
    
    Generally, $(\Theta, g_{\theta})$ will be a Pseudo-Riemannian manifold. However, if the statistical model is non-degenerate, i.e., $g_{\theta}$ is positive definite on the tangent space $T_{\theta}(\Theta)$, then $(\Theta, g_{\theta})$ forms a Riemannian manifold.  We call $(\Theta, g_{\theta})$ the Wasserstein statistical manifold. 
    \begin{proposition}
        \label{prop: metric tensor on theta}
        The metric tensor can be written as
        \begin{equation}
        g_{\theta}(\xi, \eta)=\xi^TG_W(\theta)\eta,
        \label{tensor:Gw}
        \end{equation}    
        where $G_W(\theta)\in\mathbb{R}^{d\times d}$ is a positive definite matrix and can be represented by
        \begin{equation*}
        G_W(\theta)= G_{\theta}^T A(\theta)  G_{\theta},
        \end{equation*}
        in which $A_{ij}(\theta)=\int_{\Omega} \partial_{\theta_i} \rho(x,\theta)(-\Delta_{\theta})^{-1} \partial_{\theta_j} \rho(x,\theta) dx$ and $-\Delta_{\theta}=-\nabla\cdot(\rho(x,\theta) \nabla)$. The matrix $G_{\theta}$ associates with the original metric tensor in $\Theta$ such that $\langle \dot{\theta}_1, \dot{\theta}_2\rangle_{\theta}=\dot{\theta}_1^TG_{\theta}\dot{\theta}_2$ for any $\dot{\theta}_1,\dot{\theta}_2 \in T_\theta(\Theta)$. If $\Theta$ is Euclidean space then $G_W(\theta)=  A(\theta)$.
    \end{proposition}
    \begin{proof}
        Write down the metric tensor
        \begin{align*}
        g_{\theta}(\xi, \eta)&=\int_\Omega \rho(x, \theta) \nabla \Phi_\xi(x) \cdot \nabla \Phi_\eta(x) dx\\
        & \overset{a)}{=} \int_{\Omega} \langle \nabla_{\theta} \rho(x,\theta), \xi \rangle_{\theta}\cdot \Phi_\eta(x) dx\\
        & = \int_{\Omega} \langle \nabla_{\theta} \rho(x,\theta), \xi \rangle_{\theta} (-\Delta_{\theta})^{-1} \langle \nabla_{\theta} \rho(x,\theta), \eta \rangle_{\theta} dx
        \end{align*}
        where $a)$ is due to integration by parts. Comparing the above equation with \eqref{tensor:Gw} finishes the proof.
    \end{proof}
    Given this $G_W(\theta)$, we derive the geodesic in this manifold and illustrate its connection to the geodesic in $\mathcal{P}_2(\Omega)$ as follows.
    \begin{proposition}
        The geodesics in $(\Theta, g_\theta)$ satisfies 
        \begin{equation}\label{Ham}
        \begin{cases}
        \dot \theta-G_W(\theta)^{-1}S=0\\    
        \dot S+\frac{1}{2}S^T \frac{\partial }{\partial \theta}G_W(\theta)^{-1}S=0\ 
        \end{cases}
        \end{equation}
    \end{proposition}
    \begin{proof}
        In geometry, the square of geodesic distance $d_W$ between $\rho(\cdot,\theta^0)$ and $\rho(\cdot,\theta^1)$ equals the energy functional:
        \begin{equation}
        d^2_W(\rho^0(\cdot,\theta),\rho^1(\cdot,\theta))=\inf_{\theta(t)\in C^1(0,1)}\{\int_0^1 \dot\theta(t)^T G_W(\theta) \dot\theta(t) dt\colon  \theta(0)=\theta^0, \theta(1)=\theta^1\}.
        \label{eqn:d_G}
        \end{equation} 
        The minimizer of \eqref{eqn:d_G} satisfies the geodesic equation. Let us write down the Lagrangian $L(\dot\theta,\theta)=\frac{1}{2}\dot\theta^T G_W(\theta) \dot\theta$. The geodesic satisfies the Euler-Lagrange equation 
        \begin{equation*}
        \frac{d}{dt}\nabla_{\dot\theta}L(\dot\theta, \theta)=\nabla_\theta L(\dot\theta, \theta).
        \end{equation*}
        By the Legendre transformation, 
        \begin{equation*}
        H(S,\theta)=\sup_{\dot\theta\in T_\theta(\Theta)}S^T\dot\theta-L(\dot\theta, \theta).
        \end{equation*}
        Then $S= G_W(\theta)\dot\theta$ and $H(S,\theta)=\frac{1}{2}S^TG_W(\theta)^{-1}S$. Thus we derive the Hamilton's equations
        \begin{equation*}
        \dot \theta=\partial_SH(\theta, S), \quad \dot S=-\partial_\theta H(\theta, S). 
        \end{equation*}
        This recovers \eqref{Ham}. 
    \end{proof}
    \begin{remark}
        We recall the Wasserstein geodesic equation in $(\mathcal{P}_2(\Omega), W_2)$: 
        \begin{equation*}
        \begin{cases}
        \frac{\partial\rho(t,x)}{\partial t}+\nabla\cdot(\rho(t,x)\nabla\Phi(t,x))=0\\
        \frac{\partial\Phi(t,x)}{\partial t}+\frac{1}{2}(\nabla\Phi(t,x))^2=0
        \end{cases}
        \end{equation*}    
        The above PDE pair contains both continuity equation and Hamilton-Jacobi equation. Our equation \eqref{Ham} can be viewed as the continuity equation and Hamilton-Jacobi equation on the parameter space. The difference is that when restricted to a statistical model, the continuity equation and the associated Hamilton-Jacobi equation can only flow in the probability densities constrained in $\rho(\Theta)$. 
    \end{remark}
    \begin{remark}
        \label{rem:w equal d}
        If the optimal flow $\rho_t, 0\leq t \leq 1$ in the continuity equation \eqref{BB2}
        totally lies in the probability subspace parameterized by $\theta$, then the two geodesic distances coincide: $$d_W(\theta^0,\theta^1)=W_2(\rho^0(\cdot,\theta),\rho^1(\cdot,\theta)).$$
        It is well known that the optimal transportation path between two Gaussian distributions will also be Gaussian distributions. Hence when $\rho(\cdot,\theta)$ are Gaussian measures, the above condition is satisfied. This means Gaussian is a totally geodesic submanifold. In general, $d_W$ will be different from the $L^2$ Wasserstein metric. We will demonstrate this fact in our numerical examples. 
    \end{remark}        
    \subsection{Wasserstein natural gradient}     
    Based on the Riemannian structure established in the previous section, we are able to introduce the gradient flow on the parameter space $(\Theta,g_{\theta})$. Given an objective function $R(\theta)$, the associated gradient flow will be: 
    \begin{equation*}
    \frac{d\theta}{dt}=-\nabla_gR(\theta).
    \end{equation*}
    Here $\nabla_g$ is the Riemannian gradient operator satisfying
    \begin{equation*}
    g_{\theta}(\nabla_gR(\theta),\xi)=\nabla_\theta R(\theta)\cdot \xi
    \end{equation*}
    for any tangent vector $\xi\in T_\theta\Theta$, where $\nabla_\theta$ represents the Euclidean gradient operator. 
    \begin{proposition}
        The gradient flow of function $R\in C^{1}(\Theta)$ in $(\Theta, g_{\theta})$ satisfies
        \begin{equation}
        \label{eqn:gradient flow}
        \frac{d\theta}{dt}=-G_W(\theta)^{-1}\nabla_{\theta} R(\theta).
        \end{equation}
    \end{proposition}
    \begin{proof}
        By the definition of gradient operator, 
        \begin{equation*}
        \nabla_{g}R(\theta)^TG_W(\theta)\xi    =\nabla_\theta R(\theta)\cdot \xi,
        \end{equation*}
        for any $\xi$. Thus 
        $\nabla_{g}R(\theta)=G_W(\theta)^{-1}\nabla_\theta R(\theta)$.
    \end{proof}
    
    When $R(\theta)=R(\rho(\cdot,\theta))$, i.e. the function is implicitly determined by the density $\rho(\cdot,\theta)$, the Riemannian gradient can naturally reflect the change in the probability density domain. 
    This will be expressed in our experiments by using Forward-Euler to solve the gradient flow numerically. The iteration writes
    
    \begin{equation}
    \label{eqn:wasserstein natural gd}
    \theta^{n+1}=\theta^n-\tau G_W(\theta^n)^{-1} \nabla_{\theta} R(\rho(\cdot,\theta^n)).
    \end{equation}
    This iteration of $\theta^{n+1}$ can also be understood as an approximate solution to the following problem:
    \begin{equation*}
    \mathop{\arg\min}_{\theta}\  R(\rho(\cdot,\theta))+\frac{d_W(\rho(\cdot,\theta^n),\rho(\cdot,\theta))^2}{2\tau}.
    \end{equation*}
    The approximation goes as follows. Note the Wasserstein metric tensor satisfies 
    $$d_W(\rho(\cdot, \theta+\Delta \theta),\rho(\cdot, \theta))^2=\frac{1}{2}(\Delta\theta)^T G_W(\theta) (\Delta\theta) + o\left( (\Delta \theta)^2\right)\quad \text{as} \quad \Delta \theta \to 0,$$
    and $R(\rho(\cdot,\theta+\Delta\theta))=R(\rho(\cdot,\theta))+\langle \nabla_{\theta} R(\rho(\cdot,\theta)),\Delta\theta\rangle+O((\Delta\theta)^2)$. Ignoring high-order items we obtain
    \begin{equation*}
    \theta^{n+1}=\mathop{\arg\min}_{\theta}\  \langle \nabla_{\theta} R(\rho(\cdot,\theta^n)),\theta-\theta^n\rangle+\frac{(\theta-\theta^n)^TG_W(\theta^n)(\theta-\theta^n)}{2\tau}.
    \end{equation*}
    This recovers \eqref{eqn:wasserstein natural gd}. It explains $\eqref{eqn:wasserstein natural gd}$ is the steepest descent with respect to the change of probability distributions measured by $W_2$.
    
    In fact, \eqref{eqn:wasserstein natural gd} shares the same spirit in natural gradient \cite{NG} with respect to Fisher-Rao metric. To avoid ambiguity, we call it the Fisher-Rao natural gradient. It considers $\theta^{n+1}$ as an approximate solution of
    \begin{equation*}
    \mathop{\arg\min}_{\theta}\  R(\rho(\cdot,\theta))+\frac{D_{\textrm{KL}}(\rho(\cdot,\theta)\|\rho(\cdot,\theta^n))}{\tau},
    \end{equation*}
    where $D_{\textrm{KL}}$ represents the Kullback-Leibler divergence, i.e. given two densities $p, q$ on $\Omega$, then
    \[D_{\textrm{KL}}(p\| q)=\int_\Omega p(x)\log(\frac{p(x)}{q(x)}) dx.\] 
    In our case, we replace the KL divergence by the constrained Wasserstein metric. For this reason, we call \eqref{eqn:wasserstein natural gd} the Wasserstein natural gradient descent method.
    
    \subsection{1D densities}
    In the following we concentrate on the one dimensional sample space, i.e. $\Omega=\mathbb{R}$. We show that $g_{\theta}$ exhibits an 
    explicit formula. From it, we demonstrate that when the minimization is modeled by the Wasserstein distance, namely $R(\rho(\cdot,\theta))$ is related to $W_2$, then $G_W(\theta)$ will approach the Hessian matrix of $R(\rho(\cdot,\theta))$ at the minimizer.    
    
    \begin{proposition}
        \label{thm: 1-d tensor}
        Suppose $\Omega=\mathbb{R}, \Theta=\bR^d$ is the Euclidean space, and assumption \ref{assume: elliptic} is satisfied, then the Riemannian inner product on the Wasserstein statistical manifold $(\Theta,g_\theta)$ has explicit form  
        \begin{align}
        G_W(\theta) = \int_{\bR} \frac{1}{\rho(x,\theta)}(\nabla_{\theta} F(x,\theta))^T\nabla_{\theta} F(x,\theta)dx,
        \label{eqn:one dimension G}
        \end{align}
        such that $g_{\theta}(\xi,\eta)=\langle \xi, G_W(\theta) \eta \rangle$.  
    \end{proposition}
    \begin{proof}
        When $\Omega=\mathbb{R}$, we have
        \begin{equation*}
        g_{\theta}(\xi,\eta)=\int_\bR \rho(x,\theta) \Phi'_\xi(x) \cdot \Phi'_\eta(x) dx,
        \end{equation*}
        where $\langle \nabla_{\theta} \rho(x, \theta), \xi \rangle = \left(\rho \Phi'_{\xi}(x)\right)'$ and $\langle \nabla_{\theta} \rho(x, \theta), \eta \rangle = \left(\rho \Phi'_{\eta}(x)\right)'$. \\
        Integrating the two sides yields
        \[\int_{-\infty}^y \langle \nabla_{\theta} \rho(x,\theta),\xi \rangle = \rho \Phi_{\xi}'(y).\]
        Denote by $F(y)=\int_{-\infty}^y \rho(x) dx$ the cumulative distribution function of $\rho$, then 
        \[\Phi'_{\xi}(x)=\frac{1}{\rho(x,\theta)}\langle \nabla_{\theta} F(x,\theta), \xi \rangle,\]
        and
        \begin{equation*}
        g_{\theta}(\xi,\eta)=\int_\bR \frac{1}{\rho(x,\theta)} \langle \nabla_{\theta} F(x,\theta), \xi \rangle \langle \nabla_{\theta} F(x,\theta), \eta \rangle dx.
        \end{equation*}
        This means $g_{\theta}(\xi,\eta)=\langle \xi, G_W(\theta) \eta \rangle$ and we obtain \eqref{eqn:one dimension G}. Since assumption \ref{assume: elliptic} is satisfied, this integral is well-defined.
    \end{proof}
    
    Recall the Fisher-Rao metric tensor, also known as the Fisher information matrix:
    \begin{equation*}
    \begin{aligned}
    G_F(\theta) = &\int_{\bR} \rho(x,\theta)(\nabla_{\theta} \log \rho(x,\theta))^T\nabla_{\theta} \log \rho(x,\theta)dx \\
    =&\int_{\bR} \frac{1}{\rho(x,\theta)}(\nabla_{\theta} \rho(x,\theta))^T\nabla_{\theta} \rho(x,\theta)dx,
    \end{aligned}
    \end{equation*}
    where we use the fact $\nabla_\theta\log\rho(x,\theta)=\frac{1}{\rho(x,\theta)}\nabla_\theta\rho(x,\theta)$. Compared to the Fisher-Rao metric tensor, our Wasserstein metric tensor $G_W(\theta)$ only changes the density function in the integral to the corresponding cumulative distribution function. We note the condition that $\rho$ is everywhere positive can be relaxed, for example, by assuming each component of $\nabla_{\theta} F(x,\theta)$, when viewed as a density in $\bR$, is absolutely continuous with respect to $\rho(x,\theta)$. Then we can use the associated Radon-Nikodym derivative to define the integral. This treatment is similar to the one for Fisher-Rao metric tensor \cite{IG2}.
    
    Now we turn to study the computational property of the natural gradient method with the Wasserstein metric. For standard Fisher-Rao natural gradient, it is known that when $R(\rho(\cdot,\theta))=\text{KL}(\rho(\cdot,\theta),\rho(\cdot,\theta^*))$, then
    \[\lim_{\theta \to \theta^*} G_F(\theta) = \nabla^2_{\theta} R(\rho(\cdot,\theta^*)). \]
    Hence,  $G_F(\theta)$ will approach the Hessian of $R$ at the minimizer. Regarding this, the Fisher-Rao natural gradient descent iteration
    $$ \theta^{n+1}=\theta^n-G_F(\theta^n)^{-1} \nabla_\theta R(\rho(\cdot, \theta))$$
    will be asymptotically Newton method for KL divergence related minimization. 
    
    We would like to demonstrate a similar result for the Wasserstein natural gradient. In other words, we shall show the Wasserstein natural gradient will be asymptotically Newton's method for the Wasserstein distance-related minimization. To achieve this, we start by providing a detailed description of the Hessian matrix for the Wasserstein metric in Theorem \ref{thm: hessian of R}. Throughout the following discussion, we use the notation $T'(x,\theta)$ to represent the derivative of $T$ with respect to the $x$ variable. We make the following assumption, which is needed in the proof of Theorem \ref{thm: hessian of R} to interchange the differentiation and integration.
    \begin{assumption}
        For any $\theta_0 \in \Theta$, there exists a neighborhood $N(\theta_0) \subset \Theta$, such that 
        \begin{align*}
        &\int_\Omega \max_{\theta \in N(\theta_0)}|\frac{\partial^2 F(x,\theta)}{\partial \theta_i \partial \theta_j}|\ dx < +\infty \\ 
        &\int_\Omega \max_{\theta \in N(\theta_0)} |\frac{\partial \rho(x,\theta)}{\partial \theta_i}|\ dx < +\infty \\
        &\int_\Omega \max_{\theta \in N(\theta_0)} \frac{1}{\rho(x,\theta)}|\frac {\partial F(x,\theta)}{\partial {\theta_i}}  \frac{\partial F(x,\theta)}{\partial {\theta_j}}| dx < +\infty
        \end{align*}
        \label{assume: 2}
        for each $1\leq i,j \leq d$.
    \end{assumption}
    \begin{theorem}
        Consider the statistical model $(\Omega,\Theta,\rho)$, in which $\rho(\cdot, \theta)$ is positive and $\Omega$ is a compact region in $\bR$. Suppose Assumption \ref{assume: elliptic} and \ref{assume: 2} are satisfied and $T'(x,\theta)$ is uniformly bounded for all $x$ when $\theta$ is fixed. If the objective function has the form
        \begin{equation*}
        R(\rho(\cdot,\theta))=\frac{1}{2}\left(W_2(\rho(\cdot,\theta),\rho^*)\right)^2,
        \end{equation*}
        where $\rho^*$ is the ground truth density, then
        \begin{equation}
        \nabla^2_\theta R(\rho(\cdot,\theta))= \int_\Omega (T(x,\theta)-x) \nabla^2_\theta F(x,\theta) dx + \int_\Omega \frac{T'(x,\theta)}{\rho(x,\theta)}(\nabla_{\theta} F(x,\theta))^T \nabla_\theta F(x,\theta) dx,
        \label{eqn:hessian of R}
        \end{equation}
        in which $T(\cdot,\theta)$ is the optimal transport map between $\rho(\cdot,\theta)$ and $\rho^*$, the function $F(\cdot,\theta)$ is the cumulative distribution function of $\rho(\cdot,\theta)$.
        \label{thm: hessian of R}
    \end{theorem}
    \begin{proof}
        We recall the three formulations of OT in section \ref{review} and the following facts for 1D Wasserstein distance. They will be used in the proof.
        
        \noindent(i) When $\Omega \subset \mathbb{R}$, the optimal map will have explicit formula, namely $T(x)=F_1^{-1}(F_0(x))$, where $F_0,F_1$ are cumulative distribution functions of $\rho^0,\rho^1$ respectively. Moreover, $T$ satisfies 
        \begin{equation}
        \rho^0(x)=\rho^1(T(x))T'(x).
        \label{eqn:monge ampere}
        \end{equation}
        
        \noindent(ii) The dual of linear programming \eqref{eq:Kantorovich} has the form
        \begin{equation}
        \label{eqn:Kantorovich dual}
        \max_{\phi} \int_{\Omega} \phi(x)\rho^0(x) dx + \int_{\Omega} \phi^c(x) \rho^1(x) dx,
        \end{equation}
        in which $\phi$ 
        and $\phi^c$ satisfy
        $$\phi^c(y)=\inf_{x \in \Omega} \|x-y\|^2-\phi(x).$$
        
        \noindent(iii) We have the relation $\nabla \phi(x)=2(x-T(x))$ for the optimal $T$ and $\phi$.
        
        Using the above three facts, we have            
        \[R(\rho(\cdot,\theta))=\frac{1}{2}\left(W_2(\rho(\cdot,\theta),\rho^*)\right)^2=\frac{1}{2}\int_\Omega |x-F_*^{-1}(F(x,\theta))|^2\rho(x,\theta) dx,\]
        where $F_*$ is the cumulative distribution function of $\rho^*$. We first compute $\nabla_\theta R(\rho(\cdot,\theta))$. Fix $\theta$ and assume the dual maximum is achieved by $\phi^*$ and $\phi^{c*}$: 
        \[R(\rho(\cdot,\theta))= \frac{1}{2} \int_{\Omega} \phi(x)^*\rho(x,\theta) dx + \int_{\Omega} \phi^{c*}(x) \rho^*(x) dx. \]
        Then, for any $\hat{\theta} \in \Theta$,
        \[R(\rho(\cdot,\hat{\theta})) \leq \frac{1}{2} \int_{\Omega} \phi^*(x)\rho(x,\hat{\theta}) dx + \int_{\Omega} \phi^{c*}(x) \rho^*(x) dx,\]
        and the equality holds when $\hat{\theta}=\theta$. Thus
        \[\nabla_\theta R(\rho(\cdot,\theta))=\nabla_\theta \frac{1}{2}\int_\Omega \phi^*(x) \rho(x,\theta) dx=\frac{1}{2}\int_\Omega \phi(x,\theta)\nabla_\theta \rho(x,\theta) dx,\]
        in which integration and differentiation are interchangeable due to Lebesgue dominated convergence theorem and assumption \ref{assume: 2}. The function $\phi(x,\theta)$ is the Kantorovich potential associated with $\rho(x,\theta)$.
        
        As is mentioned in (iii), $(\phi(x,\theta))'=2(x-T(x,\theta))$, which leads to 
        \[\nabla_\theta R(\rho(\cdot,\theta))= -\int_\Omega (x-T(x,\theta)) \nabla_\theta F(x,\theta) dx.\]
        Differentiation with respect to $\theta$ and interchange integration and differentiation:
        \[\nabla^2_\theta R(\rho(\cdot,\theta))= -\int_\Omega (x-T(x,\theta)) \nabla^2_\theta F(x,\theta) dx + \int_\Omega (\nabla_{\theta} T(x,\theta))^T \nabla_\theta F(x,\theta) dx. \]
        On the other hand, $T$ satisfies the following equation as in \eqref{eqn:monge ampere}:
        \begin{equation}
        \label{eqn:monge ampere theta}
        \rho(x,\theta)=\rho^*(T(x,\theta))T'(x,\theta).
        \end{equation}
        Differentiating with respect to $\theta$ and noticing that the derivative of the right hand side has a compact form: 
        \[\nabla_{\theta} \rho(x,\theta)= \left( \rho^*(T(x,\theta))\nabla_{\theta} T(x,\theta) \right)',\]
        and hence
        \[\nabla_{\theta} T(x,\theta)=\frac{1}{\rho^*(T(x,\theta))}\int_{-\infty}^x \nabla_{\theta} \rho(y,\theta) dy=\frac{\nabla_{\theta} F(x,\theta)}{\rho^*(T(x,\theta))}.\]
        Combining them together we obtain
        \begin{equation}
        \nabla^2_\theta R(\rho(\cdot,\theta))= \int_\Omega (T(x,\theta)-x) \nabla^2_\theta F(x,\theta) dx + \int_\Omega \frac{1}{\rho^*(T(x,\theta))}(\nabla_{\theta} F(x,\theta))^T \nabla_\theta F(x,\theta) dx.
        \end{equation}
        Substituting $\rho^*(T(x,\theta))$ by $\rho(x,\theta)$ and $T(x,\theta)$ based on \eqref{eqn:monge ampere theta}, we obtain \eqref{eqn:hessian of R}. Since assumption \ref{assume: elliptic} and \ref{assume: 2} are satisfied, and $T'(x,\theta)$ is uniformly bounded, the integral in \eqref{eqn:hessian of R} is well-defined.
    \end{proof}
    \begin{proposition}
        Under the condition in Theorem \ref{thm: hessian of R}, if the ground-truth density satisfies $\rho^*=\rho(\cdot,\theta^*)$ for some $\theta^* \in \Theta$, then \[\lim_{\theta \to \theta^*} \nabla^2_\theta R(\rho(\cdot,\theta))=\int_\bR \frac{1}{\rho(x,\theta^*)}(\nabla_{\theta} F(x,\theta^*))^T\nabla_{\theta} F(x,\theta^*)dx=G_W(\theta^*).\]
        \label{prop: hessian and G}
    \end{proposition}
    \begin{proof}
        When $\theta$ approaches $\theta^*$, $T(x,\theta)-x$ will go to zero and $T'$ will go to the identity. We finish the proof by the result in Theorem \ref{thm: hessian of R}.
    \end{proof}
    Proposition \ref{prop: hessian and G} explains that the Wasserstein natural gradient descent is asymptotically Newton method for the Wasserstein metric based minimization. The preconditioner $G_W(\theta)$ equals the Hessian matrix of $R$ at the ground truth. Moreover, the formula \eqref{eqn:hessian of R} contains more information than proposition \ref{prop: hessian and G}, and they can be used to find suitable Hessian-like preconditioners. For example, we can see the second term in \eqref{eqn:hessian of R} is different from $G_W(\theta)$ when $\theta \neq \theta^*$. It seems more accurate to use the term 
    \begin{equation}
    \bar{G}_W(\theta):=\int \frac{T'(x,\theta)}{\rho(x,\theta)}(\nabla_{\theta} F(x,\theta))^T\nabla_{\theta} F(x,\theta)dx
    \end{equation}
    to approximate the Hessian of $R(\rho(\cdot,\theta))$. When $\theta$ is near to $\theta^*$, $\bar{G}_W$ is likely to achieve slightly faster convergence to $\nabla^2_{\theta} R(\rho(\cdot,\theta^*))$ than $G_W$. However, the use of $\bar{G}_W(\theta)$ could also have several difficulties. The presence of $T'(x,\theta)$ limits its application to general minimization problem in space $\Theta$ which does not involve a Wasserstein metric objective function. Also, the computation of $T'(x,\theta)$ might suffer from potential numerical instability,  especially when $T$ is not smooth, which is often the case when the ground-truth density is a sum of delta functions. This fact is also expressed in our numerical examples.
    
    \section{Examples}
    \label{section:examples}
    In this section, we consider several concrete statistical models. We compute the related metric tensor $G_W(\theta)$, either explicitly or numerically, and further calculate the geodesic in the Wasserstein statistical manifold. Moreover, we test the Wasserstein natural gradient descent method in the Wasserstein distance-based inference and fitting problems \cite{Bernton2017Inference}. We show that the preconditioner $G_W(\theta)$ exhibits promising performance, leading to stable and fast convergence of the iterations. 
    We also compare our results with the Fisher-Rao natural gradient.
    \subsection{Gaussian measures}
    \label{eg: Gaussian}
    We consider the multivariate Gaussian densities $\cN(\mu,\Sigma)$ in $\bR^n$:
    \[\rho(x,\theta)=\frac{1}{\sqrt{\det(2\pi\Sigma)}}\exp\left(-\frac{1}{2}(x-\mu)^T\Sigma^{-1}(x-\mu)\right),\]
    where $\theta=(\mu,\Sigma) \in \Theta := \bR^n \times \text{Sym}^+(n,\bR)$. Here $\text{Sym}^+(n,\bR)$ is the $n\times n$ positive symmetric matrix {set, which is an open subset of the $n\times n$ symmetric matrix vector space $\text{Sym}(n,\bR)$. This implies that the tangent space at each $\Sigma$ is $\text{Sym}(n,\bR)$ with metric tensor tr($S_1S_2$) for tangent vectors $S_1,S_2$, i.e., the tangent boudle is trivial. Due to this reason, we drop the $\theta$-subscript in inner product define on $\theta$, i.e, using $\langle \nabla_{\theta} \rho(x,\theta), \xi \rangle$ here instead of $\langle \nabla_{\theta} \rho(x,\theta), \xi \rangle_{\theta}$.}  We obtain an explicit formula for $g_{\theta}$ by using definition \ref{def:metric tensor}. 
    \begin{proposition}
        \label{prop: multivariate gaussian}
        The Wasserstein metric tensor for the multivariate Gaussian model is
        \[g_{\theta}\left(\xi,\eta\right)=\langle\dot{\mu_1},\dot{\mu_2}\rangle+\mathrm{tr}(S_1\Sigma S_2),\]
        for any $\xi,\eta \in T_{\theta}\Theta$. Here $\xi=(\dot{\mu}_1,\dot{\Sigma_1})$ and $\eta=(\dot{\mu}_2,\dot{\Sigma_2})$, in which $\dot{\mu}_1,\dot{\mu}_2 \in \bR^n, \dot{\Sigma_1}, \dot{\Sigma_2} \in \mathrm{Sym}(n,\bR)$, and the symmetric matrix $S_1,S_2$ satisfy $\dot{\Sigma}_1=S_1\Sigma+\Sigma S_1, \dot{\Sigma}_2=S_2\Sigma+\Sigma S_2$. 
    \end{proposition}
    
    \begin{proof}
        First we examine the elliptic equation in definition \ref{def:metric tensor} has the solution explained in Proposition \ref{prop: multivariate gaussian}. Write down the equation
        \[\langle \nabla_{\theta} \rho(x,\theta), \xi \rangle = -\nabla \cdot (\rho(x,\theta) \nabla \Phi_{\xi}(x)).\]
        By some computations we have
        \begin{equation*}
        \begin{aligned}
        &\langle \nabla_{\theta} \rho(x,\theta), \xi \rangle=\langle \nabla_{\mu} \rho(x,\theta), \dot{\mu}_1\rangle + \text{tr}(\nabla_{\Sigma} \rho(x,\theta)\dot{\Sigma}_1),\\
        &\langle \nabla_{\mu} \rho(x,\theta), \dot{\mu}_1\rangle= \dot{\mu}_1^T\Sigma^{-1}(x-\mu)\cdot \rho(x,\theta), \\
        &\text{tr}(\nabla_{\Sigma} \rho(x,\theta)\dot{\Sigma}_1)=-\frac{1}{2}\left(\text{tr}(\Sigma^{-1}\dot{\Sigma}_1)-(x-\mu)^T\Sigma^{-1}\dot{\Sigma}_1\Sigma^{-1}(x-\mu)\right) \cdot \rho(x,\theta), \\
        & -\nabla \cdot (\rho \nabla \Phi_{\xi}(x))=(\nabla \Phi_{\xi}(x))\Sigma^{-1}(x-\mu) \cdot \rho(x,\theta)-\rho(x,\theta)\Delta \Phi_{\xi}(x).
        \end{aligned}
        \end{equation*}
        Observing these equations, we let $\nabla \Phi_{\xi} (x)=(S_1(x-\mu)+\dot{\mu}_1)^T$, and $\Delta \Phi_{\xi}=\text{tr}(S_1)$, where $S_1$ is a symmetric matrix to be determined. By comparison of the coefficients and the fact $\dot{\Sigma}_1$ is symmetric, we obtain
        \[\dot{\Sigma}_1=S_1\Sigma+\Sigma S_1.\]
        Similarly $\nabla \Phi_{\eta} (x)=(S_2(x-\mu)+\dot{\mu}_2)^T$ and 
        \[\dot{\Sigma}_2=S_2\Sigma+\Sigma S_2.\]
        Then 
        \begin{align*}
        g_{\theta}(\xi, \eta)&=\int \rho(x, \theta) \nabla \Phi_\xi(x) \cdot \nabla \Phi_\eta(x) dx\\
        &=\langle\dot{\mu_1},\dot{\mu_2}\rangle+\text{tr}(S_1\Sigma S_2).
        \end{align*}
        It is easy to check $\Phi$ satisfies the condition \eqref{eqn: regularity of phi, whole space} and the uniqueness is guaranteed.
    \end{proof}
    For Gaussian distributions, the above derived metric tensor has already been revealed in \cite{GW}\cite{WM}\cite{BuresWasserstein}. Our calculation shows that it is a particular formulation of $G_W$. {Hence, this demonstrates our defined submanifold geometry is totally geodesic in the Gaussian case}. In the following, we turn to several one-dimensional non-Gaussian distributions and illustrate the metric tensor and the related geodesics, gradient flow numerically. {We will see in general, the Wasserstein statistical manifold defined here will not be totally geodesic.}
    \subsection{Mixture model}We consider a generalized version of the Gaussian distribution, namely the Gaussian mixture model. For simplicity we assume there are two components, i.e. $a\cN(\mu_1,\sigma_1)+(1-a)\cN(\mu_2,\sigma_2)$ with density functions:
    \[\rho(x,\theta)=\frac{a}{\sigma_1 \sqrt{2\pi}}e^{-\frac{(x-\mu_1)^2}{2\sigma_1^2}}+\frac{1-a}{\sigma_2 \sqrt{2\pi}}e^{-\frac{(x-\mu_2)^2}{2\sigma_2^2}},\]
    where $\theta=(a,\mu_1,\sigma_1^2,\mu_2,\sigma_2^2)$ and $a\in [0,1]$.
    
    We first compute the geodesics of Wasserstein statistical manifold numerically. Set $\theta^0=(0.3,-3,0.5^2,-5,0.4^2)$ and $\theta^1=(0.6,7,0.4^2,5,0.3^2)$. Their density functions are shown in Figure \ref{fig:mixture density}.
    \begin{figure}[ht]
        \centering
        \includegraphics[width=10cm]{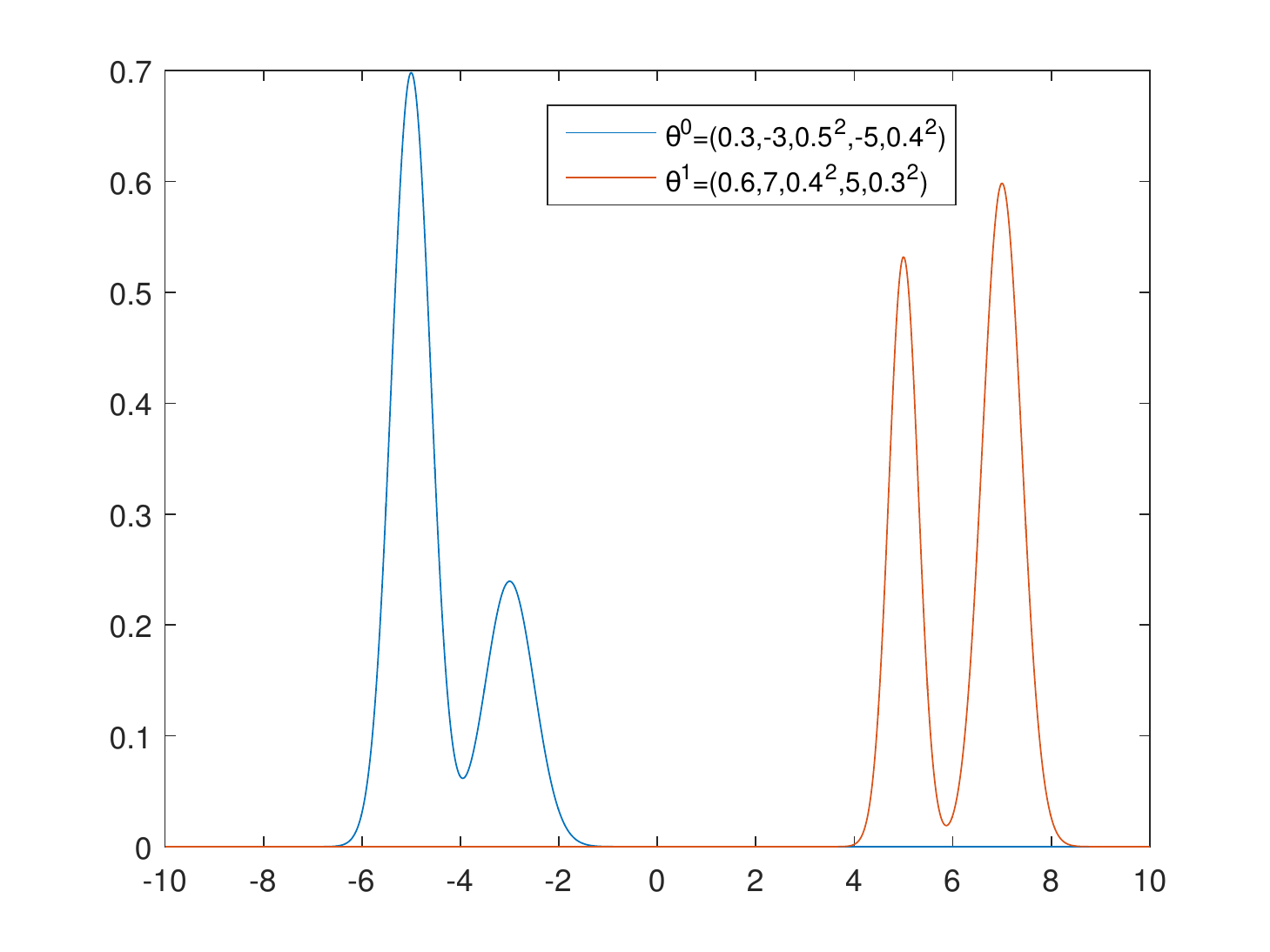}
        \caption{Densities of Gaussian mixture distribution}
        \label{fig:mixture density}
    \end{figure}
    
    To compute the geodesics in the Wasserstein statistical manifold, we solve the optimal control problem in \eqref{eqn:d_G} numerically via a direct method. Discretize the problem as
    \begin{align*}
    \min_{\theta_i, 1\leq i \leq N-1} 
    N\sum_{i=0}^{N-1} (\theta_{i+1}-\theta_i)^TG_W(\theta_i)(\theta_{i+1}-\theta_i),
    \end{align*}     
    where $\theta_0=\theta^0,\theta_N=\theta^1$ and the discrete time step-size is $1/N$. We use coordinate descent method, i.e. applying gradient on each $\theta_i, 1 \leq i \leq N-1$ alternatively till convergence.
    
    The geodesics in the whole density space is obtained by first computing the optimal transportation map $T$, using the explicit formula in one dimension $Tx=F_1^{-1}(F_0(x))$. Here $F_0,F_1$ are the cumulative distribution functions of $\rho^0$ and $\rho^1$. Then the geodesic probability densities satisfies $\rho(t,x)=(tT+(1-t)I)\#\rho_0(x)$ for $0\leq t\leq 1$, where $\#$ is the push forward operator. The result is shown in Figure \ref{fig: geodesic our metric, mixture}.
    
    \begin{figure}[ht]
        \centering
        \includegraphics[width=7cm]{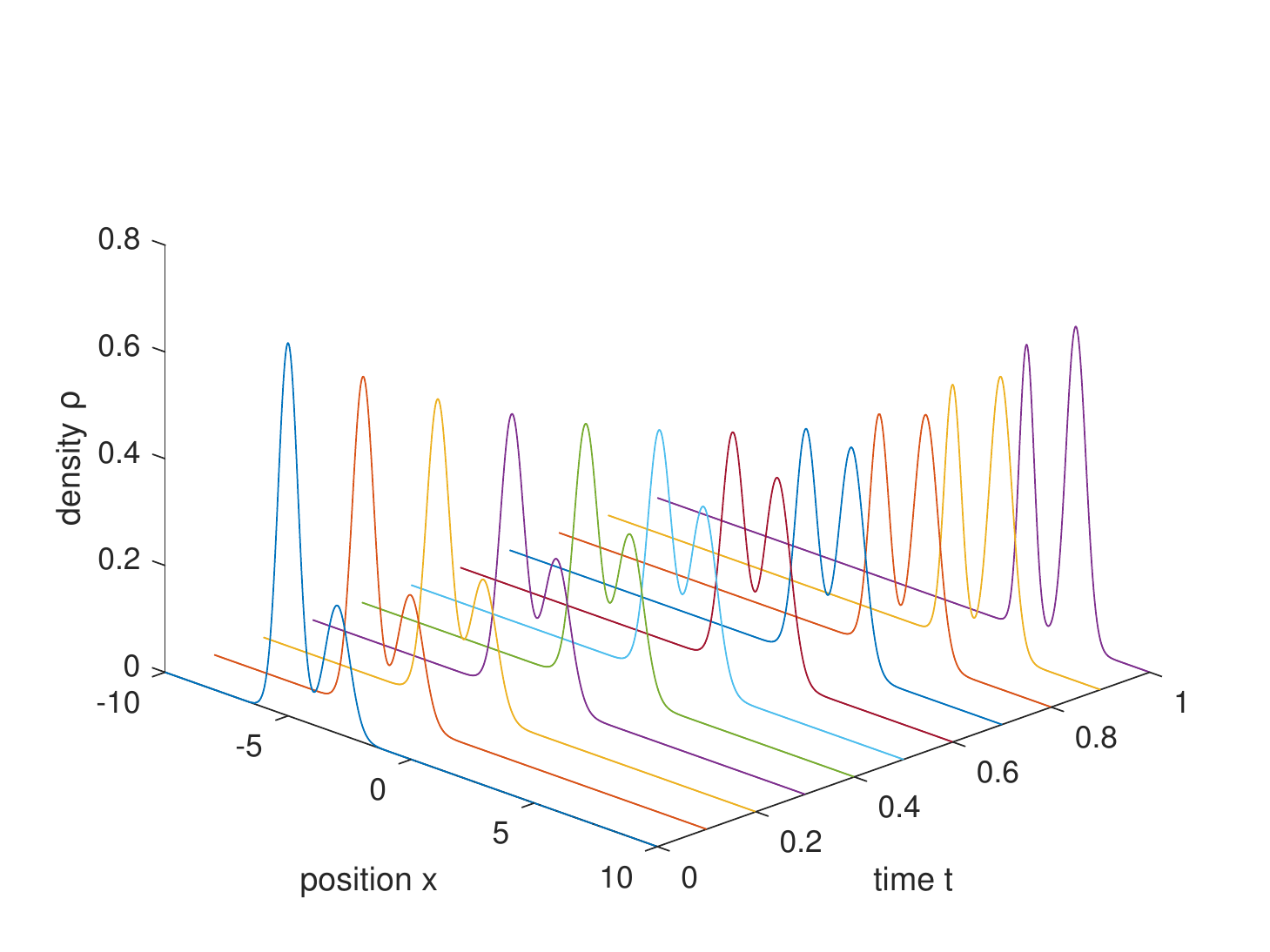}
        \includegraphics[width=7cm]{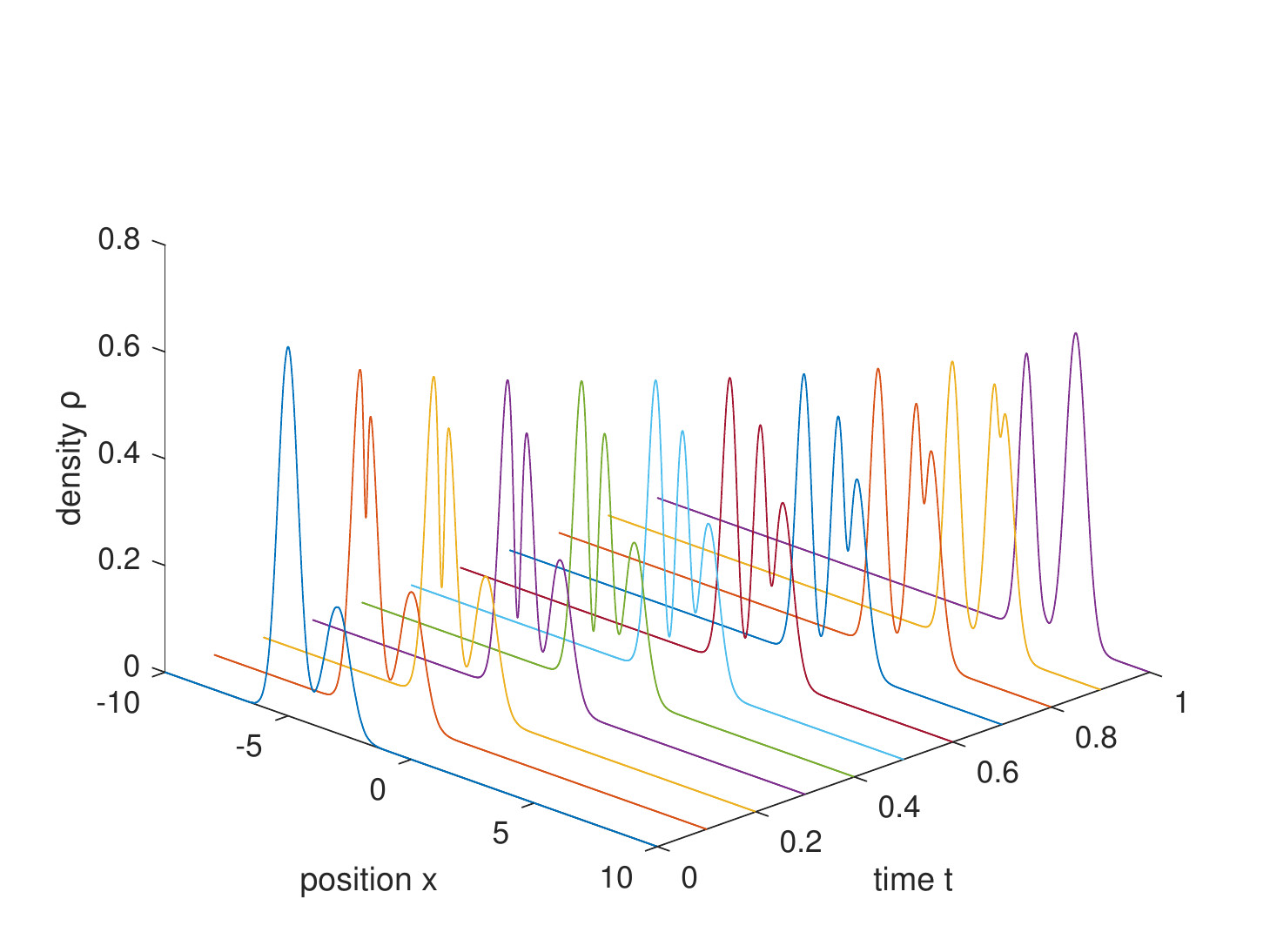}
        \caption{Geodesic of Gaussian mixtures; left: in the Wasserstein statistical manifold; right: in the whole density space}
        \label{fig: geodesic our metric, mixture}
    \end{figure}
    
    Figure \ref{fig: geodesic our metric, mixture} demonstrates that the geodesics in the whole density manifold does not lie in the sub-manifold formed by the mixture distribution, and thus the distance $d_W$ differs from the $L^2$ Wasserstein metric. Hence the optimal transport in the whole density space destroys the geometric shape during its path, which is not a desired property when we perform transportation. 
    
    Next, we test the Wasserstein natural gradient method in optimization. Consider the Gaussian mixture fitting problem: given $N$ data points $\{x_i\}_{i=1}^N$ obeying the distribution $\rho(x;\theta^1)$ (unknown), we want to infer $\theta^1$ by using these data points, which leads to a minimization as:
    \[\min_{\theta} d\left(\rho(\cdot;\theta),\frac{1}{N}\sum_{i=1}^N \delta_{x_i}(\cdot)\right), \] where $d$ are certain distance functions on probability space. If we set $d$ to be KL divergence, then the problem will correspond to the maximum likelihood estimate. Since $\rho(x;\theta^1)$ has small compact support, using KL divergence is risky and needs very good initial guess and careful optimization. Here we use the $L^2$ Wasserstein metric instead and set $N=10^3$. We truncate the distribution in $[-r,r]$ for numerical computation. We choose $r=15$. The Wasserstein metric is effectively computed by the explicit formula
    \begin{equation*}
    \frac{1}{2}\left( W_2(\rho(\cdot;\theta),\frac{1}{n}\sum_{i=1}^N \delta_{x_i})\right)^2=\frac{1}{2}\int_{-r}^r |x-T(x)|^2 \rho(x;\theta) dx,
    \end{equation*}
    where $T(x)=F_{em}^{-1}(F(x))$ and $F_{em}$ is the cumulative distribution function of the empirical distribution. The gradient with respect to $\theta$ can be computed through
    \begin{equation*}
    \nabla_{\theta} (\frac{1}{2}W^2) = \int_{-r}^r \phi(x) \nabla_\theta \rho(x;\theta) dx, 
    \end{equation*}
    where $\phi(x)=\int_{-r}^x (y-T(y)) dy$ is the Kantorovich potential. The derivative $\nabla_{\theta} \rho(x;\theta)$ is obtained by numerical differentiation. We perform the following five iterative algorithms to solve the optimization problem: 
    \begin{align*}
    \label{iterative schemes}
    &\text{Gradient descent (GD)}: \quad\theta_{n+1}=\theta_n-\tau \nabla_{\theta} (\frac{1}{2}W^2)|_{\theta_n}\\
    &\text{GD with diag-preconditioning}: \quad\theta_{n+1}=\theta_n-\tau P^{-1}\nabla_{\theta} (\frac{1}{2}W^2)|_{\theta_n}\\
    &\text{Wasserstein GD}: \quad\theta_{n+1}=\theta_n-\tau G_W(\theta_n)^{-1} \nabla_{\theta} (\frac{1}{2}W^2)|_{\theta_n} \tag{A}\\
    &\text{Modified Wasserstein GD}: \quad\theta_{n+1}=\theta_n-\tau \left(\bar{G}_W(\theta_n)\right)^{-1} \nabla_{\theta} (\frac{1}{2}W^2)|_{\theta_n}\\
    &\text{Fisher-Rao GD}: \quad\theta_{n+1}=\theta_n-\tau G_F(\theta_n)^{-1} \nabla_{\theta} (\frac{1}{2}W^2)|_{\theta_n}\\
    \end{align*}
    We consider the diagonal preconditioning because the scale of parameter $a$ is very different from $\mu_i,\sigma_i$, $1\leq i \leq 2$. The diagonal matrix $P$ is set to be $\text{diag}(40,1,1,1,1)$.  We choose the initial step-size $\tau=1$ with line search such that the objective value is always decreasing. The initial guess $\theta=\theta^0$. Figure \ref{fig: mixture fitting wasserstein} shows the experimental results.
    
    \begin{figure}[ht]
        \centering
        \includegraphics[width=10cm]{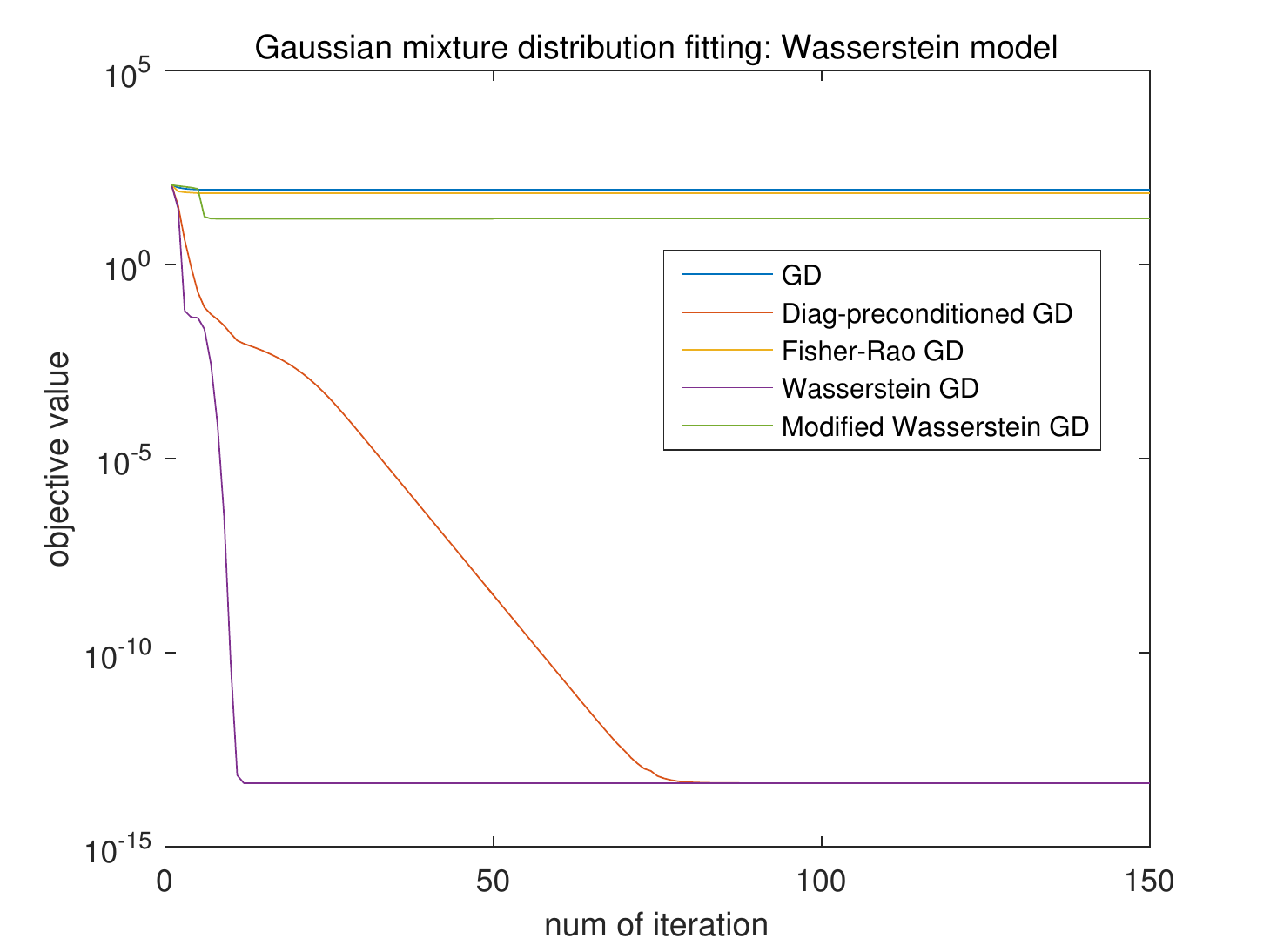}
        \caption{objective value}
        \label{fig: mixture fitting wasserstein}
    \end{figure}
    
    From the figure, it is seen that the Euclidean gradient descent fails to converge. We observed that during iterations, the parameter $a$ goes very fast to $1$ and then stop updating anymore. This is due to the ill-conditioned nature of the problem, in the sense that the scale of parameter differs drastically. If we use the diagonal matrix $P$ to perform preconditioning, then it converges after approximately 70 steps. If we use Wasserstein gradient descent, then the iterations converge very efficiently, taking less than ten steps. This demonstrates that $G_W(\theta)$ is well suited for the Wasserstein metric minimization problems, exhibiting very stable behavior. It can automatically detect the parameter scale and the underlying geometry. As a comparison, Fisher-Rao gradient descent fails, which implies $G_F(\theta)$ is not suitable for this Wasserstein metric modeled minimization. 
    
    The modified Wasserstein gradient descent does not converge because of the numerical instability of $T'$ and further $\bar{G}_W(\theta)$ in the computation. In the next example with lower dimensional parameter space, however, we will see that $\bar{G}_W(\theta)$ performs better than $G_W(\theta)$. This implies $\bar{G}_W(\theta)$, if computed accurately, might achieve a smaller approximation error to the Hessian matrix. Nevertheless, the difference is very slight, and since the matrix $\bar{G}_W(\theta)$ can only be applied to the Wasserstein modeled problem, we tend to believe that $G_W(\theta)$ is a better preconditioner. 
    
    \subsection{Gamma distribution}
    
    Consider gamma distribution $\Gamma(\alpha,\beta)$, which has the probability density function
    \[\rho(x;\alpha,\beta)=\frac{\beta^{\alpha}x^{\alpha-1}e^{-\beta x}}{\Gamma(\alpha)}.\]
    Set $\theta=(\alpha,\beta)$ and $\theta^0=(2,3), \theta^1=(20,2)$. 
    Their density functions are shown in Figure \ref{fig: gamma density}.
    \begin{figure}[ht]
        \centering
        \includegraphics[width=10cm]{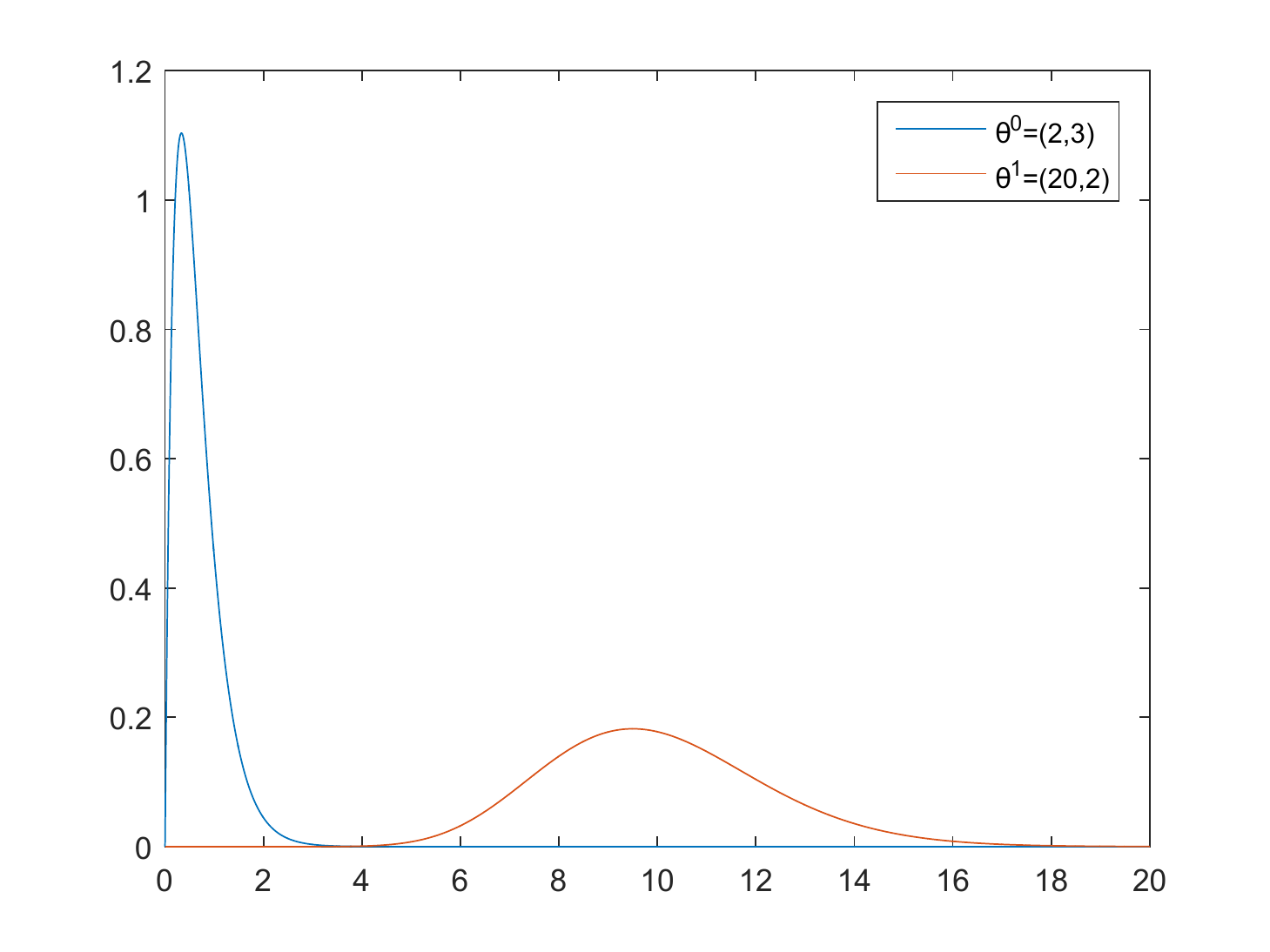}
        \caption{Gamma density functions}
        \label{fig: gamma density}
    \end{figure}
    
    We compute the related geodesics in the Wasserstein statistical manifold and the whole density space, respectively. The results are presented in Figure \ref{fig: geodesic our metric}. We can see that these two do not differ very much. This means the optimal transport in the whole space could nearly keep the gamma distribution shape along with transportation.
    \begin{figure}[ht]
        \centering
        \includegraphics[width=7cm]{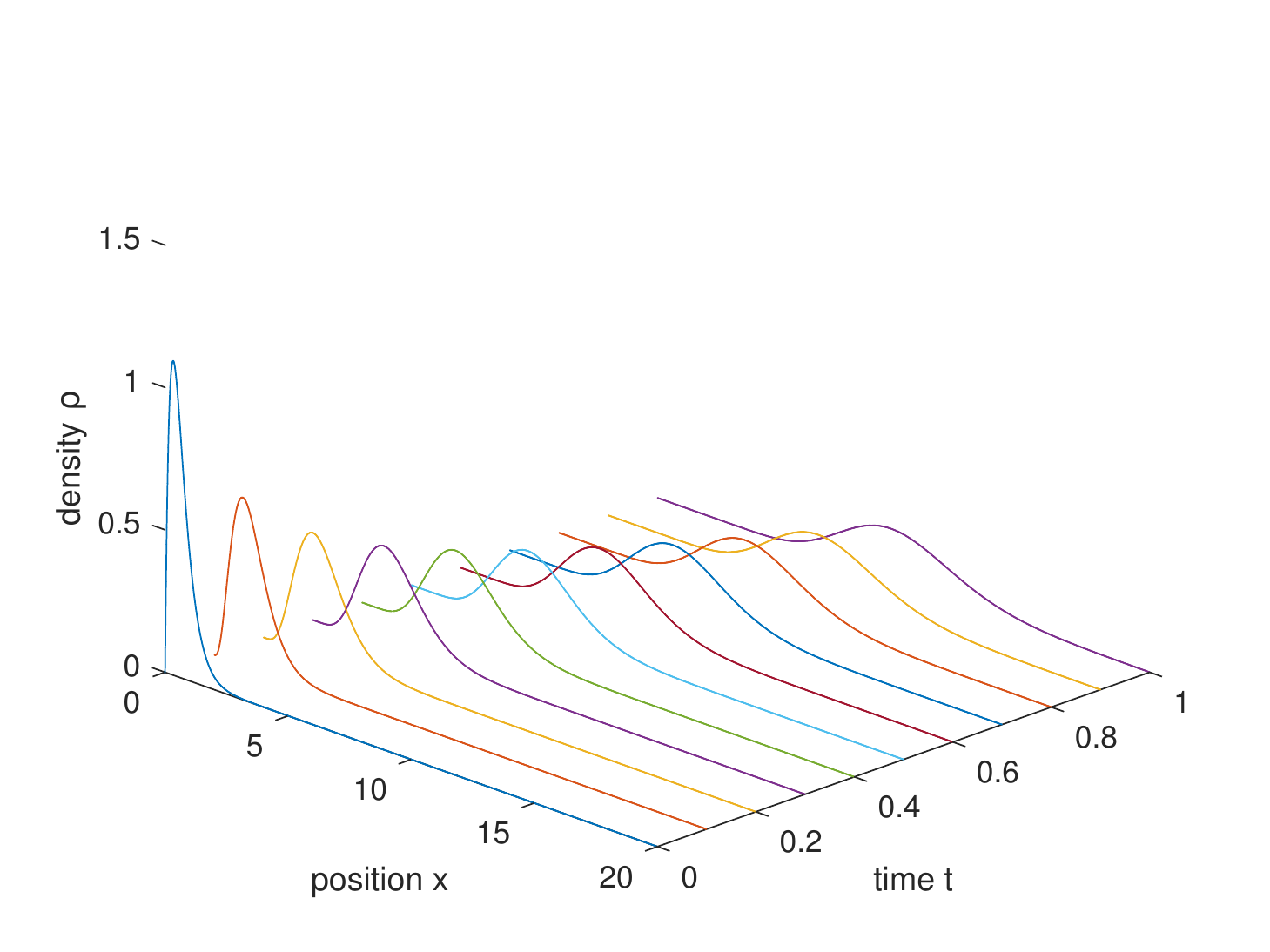}
        \includegraphics[width=7cm]{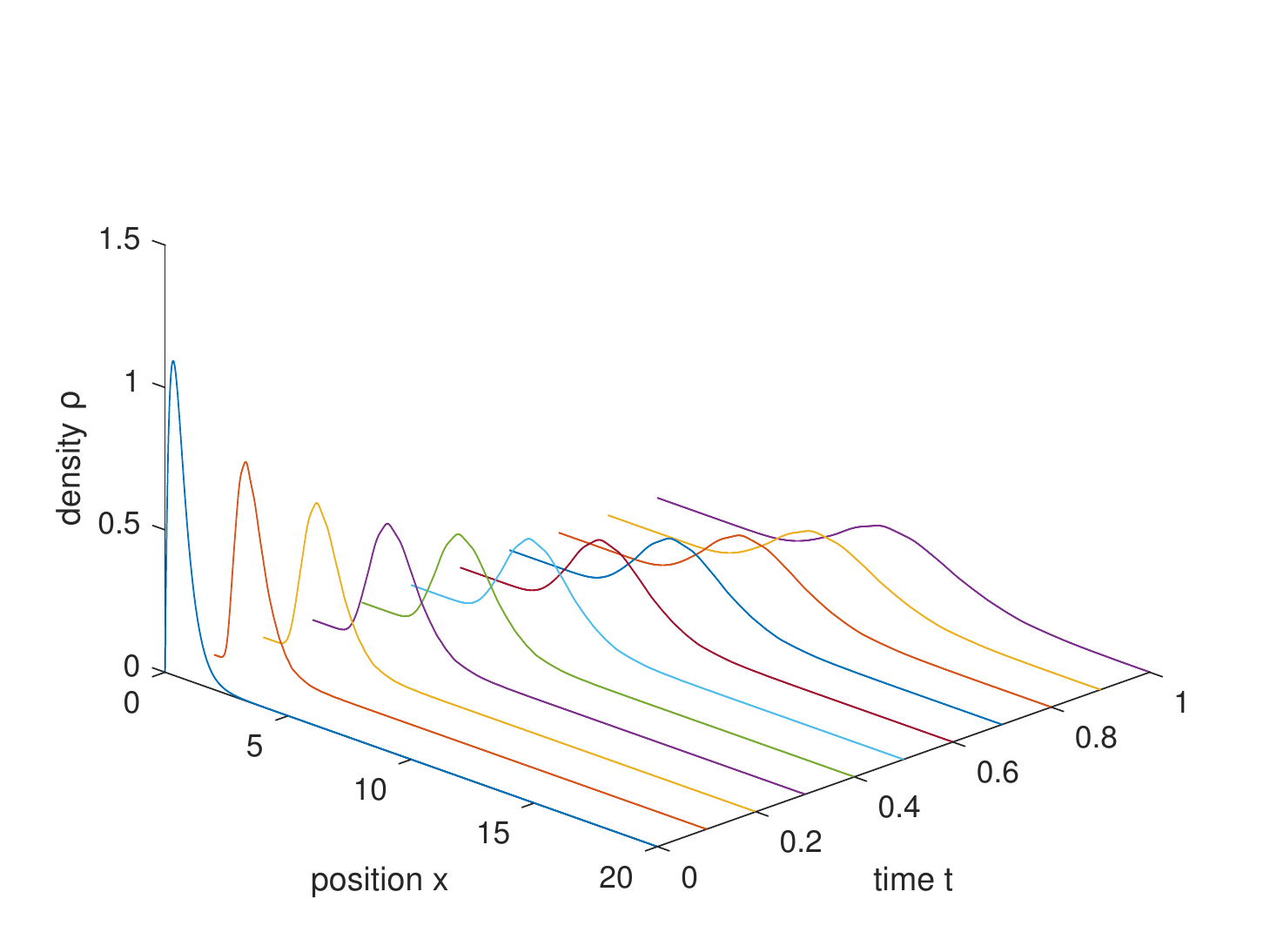}
        \caption{Geodesic of Gamma distribution; left: in the Wasserstein statistical manifold; right: in the whole density space}
        \label{fig: geodesic our metric}
    \end{figure}
    
    Then, we consider the gamma distribution fitting problem. The model is similar to the one in the mixture examples, except that the parameterized family changes. The minimization problem is:
    \[\min_{\theta} \frac{1}{2}\left( W_2(\rho(\cdot;\theta),\frac{1}{N}\sum_{i=1}^N \delta_{x_i})\right)^2,\]
    where $x_i \sim \rho(\cdot,\theta^1)$ and we set $N=10^3$. The initial guess is $\theta=\theta^0$. Convergence results are presented in Figure \ref{fig: gamma fitting}. 
    \begin{figure}[ht]
        \centering
        \includegraphics[width=10cm]{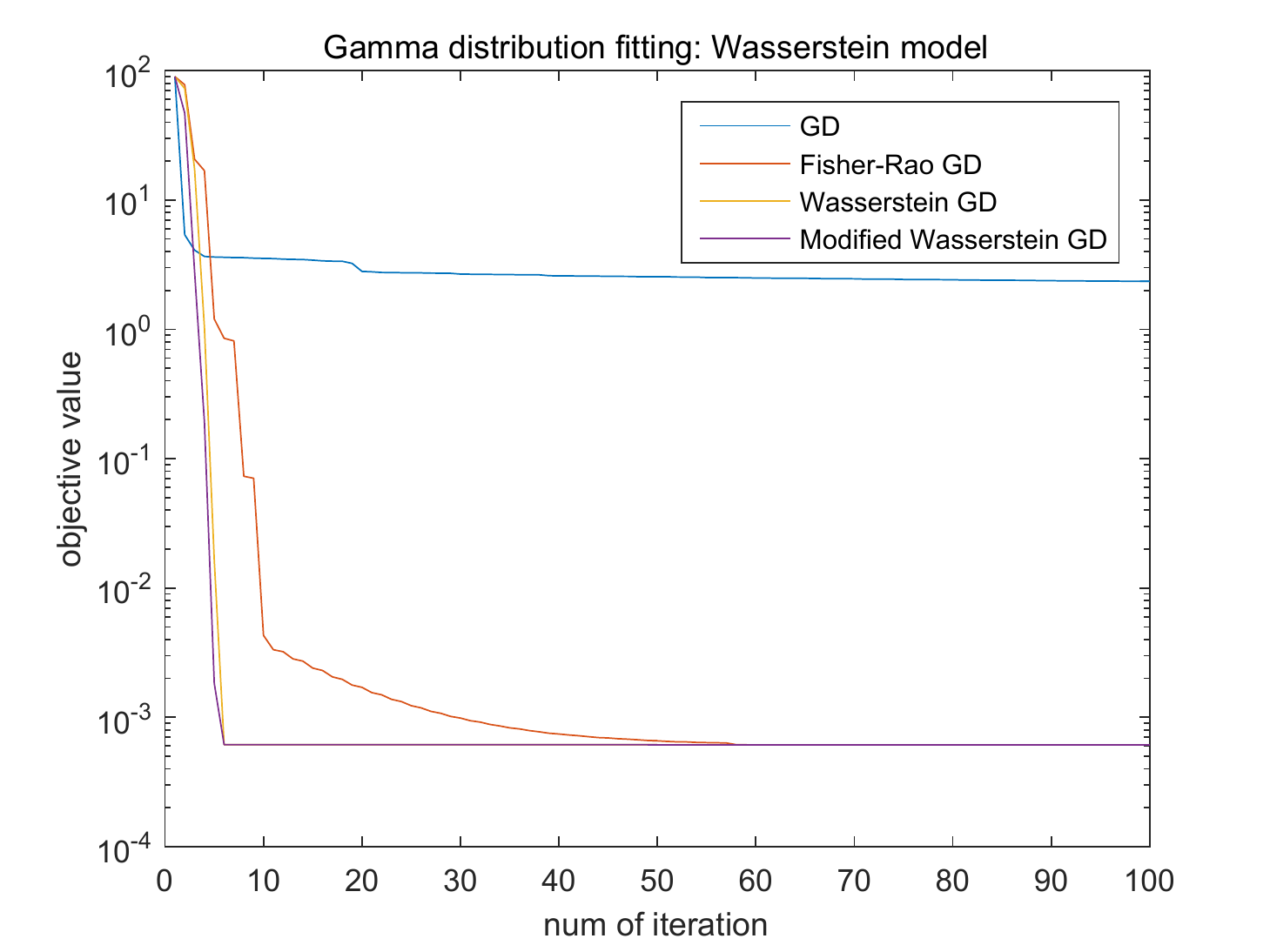}
        \caption{objective value}
        \label{fig: gamma fitting}
    \end{figure}
    
    The figure shows that the Euclidean gradient descent method takes a very long time to reach convergence, while Wasserstein GD and its modified version needs less than ten steps, with the Fisher-Rao GD taking around 50 steps. This comparison demonstrates the efficiency of the Wasserstein natural gradient in this Wasserstein metric modeled optimization problems. As is mentioned in the previous example, the difference between using $G_W$ and $\bar{G}_W$ is tiny. Since $\bar{G}_W$ fails in the mixture example, we conclude that $G_W$, the Wasserstein gradient descent, will be a more stable choice for preconditioning.
    
    \section{Numerical Comparison between Fisher-Rao and Wasserstein natural gradient}
    \label{sec: comparison FR and W}
    In the previous section, we provide several statistical estimation examples, in which the Wasserstein metric is served as the loss function, to investigate the performance of different kinds of gradient descent algorithms. Among the five iterative schemes in (\ref{iterative schemes}), we observe that the Wasserstein natural gradient outperforms the others. This demonstrates our theoretical arguments before. However, these discussions are all constrained to the Wasserstein metric loss function case and limited to the occasions that the ground truth density lies in the parametric family, i.e., the parametrization is well-specified, which may not be true in practice. Thus, it would be interesting to see the performance of these algorithms when different loss functions are used (say, KL divergence or Maximum Likelihood Estimation) and when the actual density lies outside the parametric family. This is the main target of the present section.
    
    Here, we consider three different iterative rules, namely Gradient Descent (GD), Wasserstein GD, and Fisher-Rao GD. In general, the loss is denoted by $d(\rho(\cdot,\theta),\rho^*(\cdot))$. Then, the iterative schemes write
    \begin{align*}
    \label{iterative B}
    &\text{Gradient descent (GD)}: \quad\theta_{n+1}=\theta_n-\tau \nabla_{\theta} d(\rho(\cdot,\theta),\rho^*(\cdot))|_{\theta_n}\\
    &\text{Wasserstein GD}: \quad\theta_{n+1}=\theta_n-\tau G_W(\theta_n)^{-1} \nabla_{\theta} d(\rho(\cdot,\theta),\rho^*(\cdot))|_{\theta_n} \tag{B}\\
    &\text{Fisher-Rao GD}: \quad\theta_{n+1}=\theta_n-\tau G_F(\theta_n)^{-1} \nabla_{\theta} d(\rho(\cdot,\theta),\rho^*(\cdot))|_{\theta_n}\\
    \end{align*}
    We omit the GD with diag-preconditioning in (\ref{iterative schemes}) because it would be very difficult to come up with a general diagonal preconditioning rule for arbitrary problems. The modified Wasserstein GD is also omitted because in previous examples it behaves less stable compared to the un-modified version and more crucially, it is only applicable to the Wasserstein metric loss based problems.
    
    The experiment is designed as follows: we set $d$ to be the Wasserstein metric and KL divergence (or equivalently, Maximum Likelihood Estimate), respectively. The parametric density $\rho(\cdot,\theta)$ is the two-component Gaussian mixture
    \begin{equation*}
    \label{eqn: mixture parameter}
    \rho(x,\theta)=\frac{1}{1+\exp(a)}\frac{1}{\sigma_1 \sqrt{2\pi}}e^{-\frac{(x-\mu_1)^2}{2\sigma_1^2}}+\frac{1}{1+\exp(-a)}\frac{1}{\sigma_2 \sqrt{2\pi}}e^{-\frac{(x-\mu_2)^2}{2\sigma_2^2}},
    \end{equation*}
    where $\theta=(a,\mu_1,\sigma_1^2,\mu_2,\sigma_2^2)$. 
    The target $\rho^*(\cdot)$ is the empirical distribution of data samples generated from two-component Gaussian mixture distribution (well-specified case) or Laplace distribution (misspecified case, such that $\rho^*$ is not in the parametric family) with density
    \begin{equation*}
    \label{eqn:Laplace parameter}
    \rho^*(x,\mu,b)=\frac{1}{2b}\exp(-\frac{|x-\mu|}{b}).
    \end{equation*}
    
    The number of data samples is determined later. We randomly pick some values of the parameters to generate $\rho(x,\theta^0)$ and $\rho^*$, where $\theta^0$ is the initial parameter value. We draw a certain number of samples (determined later) to form the empirical data and construct the loss function, either the Wasserstein metric or the Maximum Likelihood Estimate. Then, we apply the three iterative schemes in (\ref{iterative B}) according to the following rule: the step size $\tau$ is chosen to monotonically decrease the objective value along with the iteration (by line search: $\tau \to \tau/2$ with initial stepsize $1$). The iteration is stopped if
    \begin{itemize}
        \item $\|\nabla_{\theta} d(\rho(\cdot,\theta),\rho^*(\cdot))\| \leq \epsilon$ where we set $\epsilon=10^{-1}$; or
        \item step size is smaller than $\delta$, but line search still fails, where we set $\delta=10^{-4}$; or
        \item iteration step is larger than $N$; where we set $N=200$
    \end{itemize}
    In the stopping rule, the first case corresponds to the iteration converges. The second case may happen when the metric tensor is poorly conditioned, or the iteration jumps out of the computation region. The third case is for computational efficiency consideration to prevent too many iterations. We collect the statistical behaviors of these iterations for many differently generated $\rho(x,\theta)$ and $\rho^*$. We classify the results based on their final objective values and calculate the mean value and standard variance. We also record the number of iterations taken in the process when different methods are used.
    
    We present the experimental results in the following two subsection, from well-specified case to misspecified case.
    \subsection{Well-specified case}
    In the well-specified case, we set $\rho(\cdot,\theta_0)$ (initial point of the iteration) and $\rho^*$ both from two-component mixture family, parametrized in the following way:
    \[\rho(x,\theta)=\frac{1}{1+\exp(a)}\frac{1}{\sigma_1 \sqrt{2\pi}}e^{-\frac{(x-\mu_1)^2}{2\sigma_1^2}}+\frac{1}{1+\exp(-a)}\frac{1}{\sigma_2 \sqrt{2\pi}}e^{-\frac{(x-\mu_2)^2}{2\sigma_2^2}},\]
    where $\theta=(a,\mu_1,\sigma_1^2,\mu_2,\sigma_2^2)$.
    
    We choose $a$ uniformly distributed in $[-2,2]$, $\mu_1,\mu_2$ uniformly in $[-10,10]$, $\sigma_1^2,\sigma_2^2$ uniformly in $[1,11]$ to generate $\rho(\cdot,\theta^0)$ and $\rho^*$. We collect $200$ samples from $\rho^*$, whose empirical distribution will be the target data we use in optimization. The loss function is chosen as the Wasserstein metric and Maximum Likelihood Estimate, respectively. Given the loss function, we perform the three iterative methods. 100 trials of different $\rho(\cdot,\theta^0)$ and $\rho^*$ are tested. We output the mean / standard derivation (std) of the final objective values and number of total iteration steps in the 100 trials. We also plot the histogram of the final objective (obj) value for the three methods for better comparisons. Table \ref{table:W2 loss, well} and Figure \ref{hist: MLE, well} are for the Wasserstein metric loss, while Table \ref{table: MLE, well} and Figure \ref{hist: MLE, well} are for the Maximum Likelihood Estimation case.
    \begin{table}[ht]
        \centering
        \begin{tabular}{llll}
            \hline
            & Wasserstein GD & Fisher-Rao GD & GD       \\ \hline
            obj mean        & 0.0378          & 0.0489        & 0.1543   \\ \hline
            obj std         & 0.0208        & 0.0304         & 0.4251   \\ \hline
            iterations mean & 5.0842         & 9.6947      &  47.6105 \\ \hline
            iterations std  & 1.7052          & 5.5373       & 57.4925 \\ \hline
        \end{tabular}
        \caption{Wasserstein metric loss function: well-specified case}
        \label{table:W2 loss, well}
    \end{table}
    
    \begin{figure}[ht]
        \centering
        \includegraphics[width=8cm]{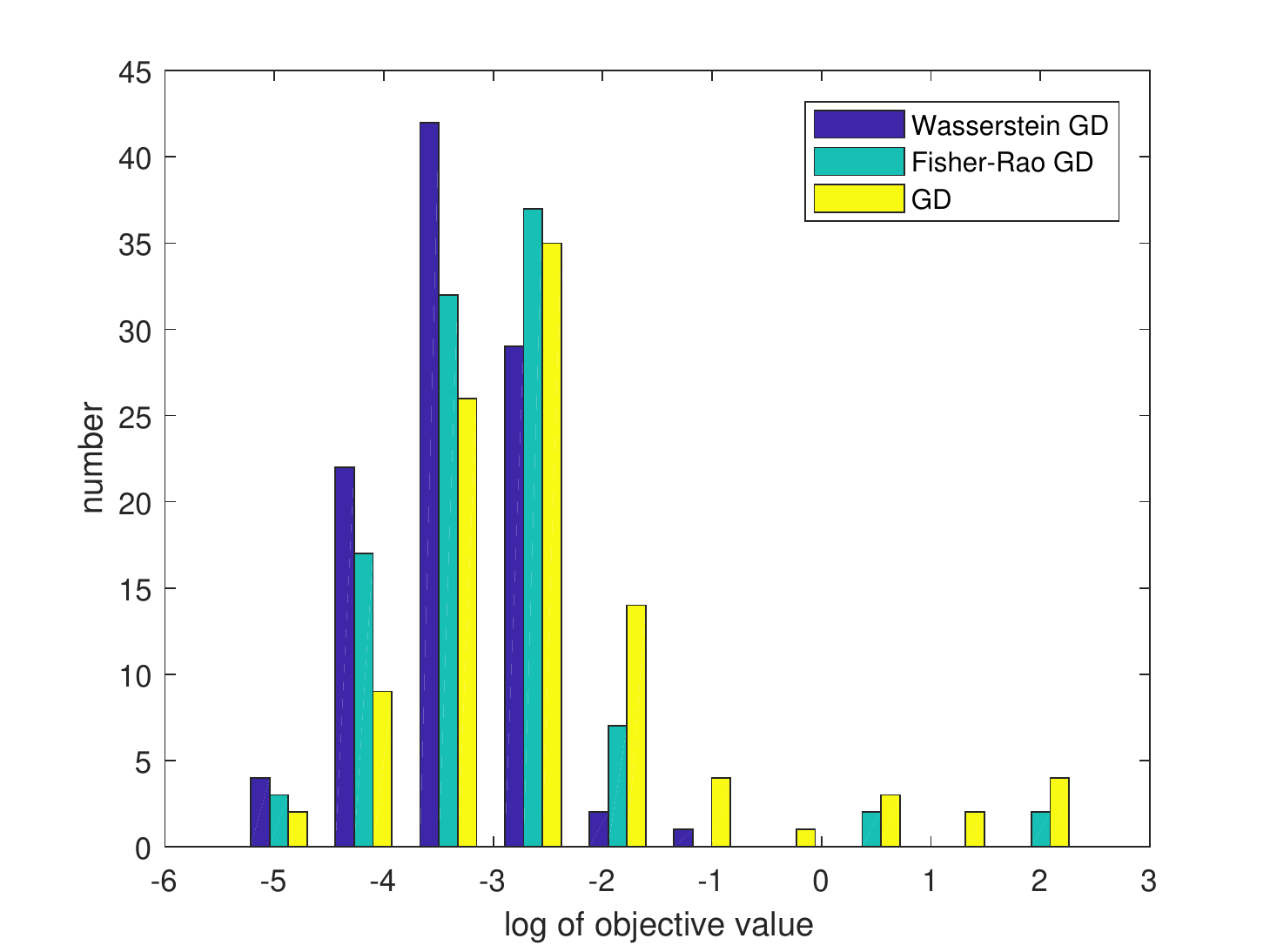}
        \caption{Histogram of obj (Wasserstein metric loss function: well-specified case)}
        \label{hist:W2 loss, well}
    \end{figure}
    
    \begin{table}[ht]
        \centering
        \begin{tabular}{llll}
            \hline
            & Wasserstein GD & Fisher-Rao GD & GD       \\ \hline
            obj mean        & 2.5613          & 2.4265        & 2.4119  \\ \hline
            obj std         & 0.4290        & 0.3548         & 0.2686   \\ \hline
            iterations mean & 102.1684        & 3.9053     &  15.4632 \\ \hline
            iterations std  & 84.5325       & 0.9461       & 6.2481 \\ \hline
        \end{tabular}
        \caption{Maximum Likelihood Estimate: well-specified case}
        \label{table: MLE, well}
    \end{table}
    
    \begin{figure}[ht]
        \centering
        \includegraphics[width=8cm]{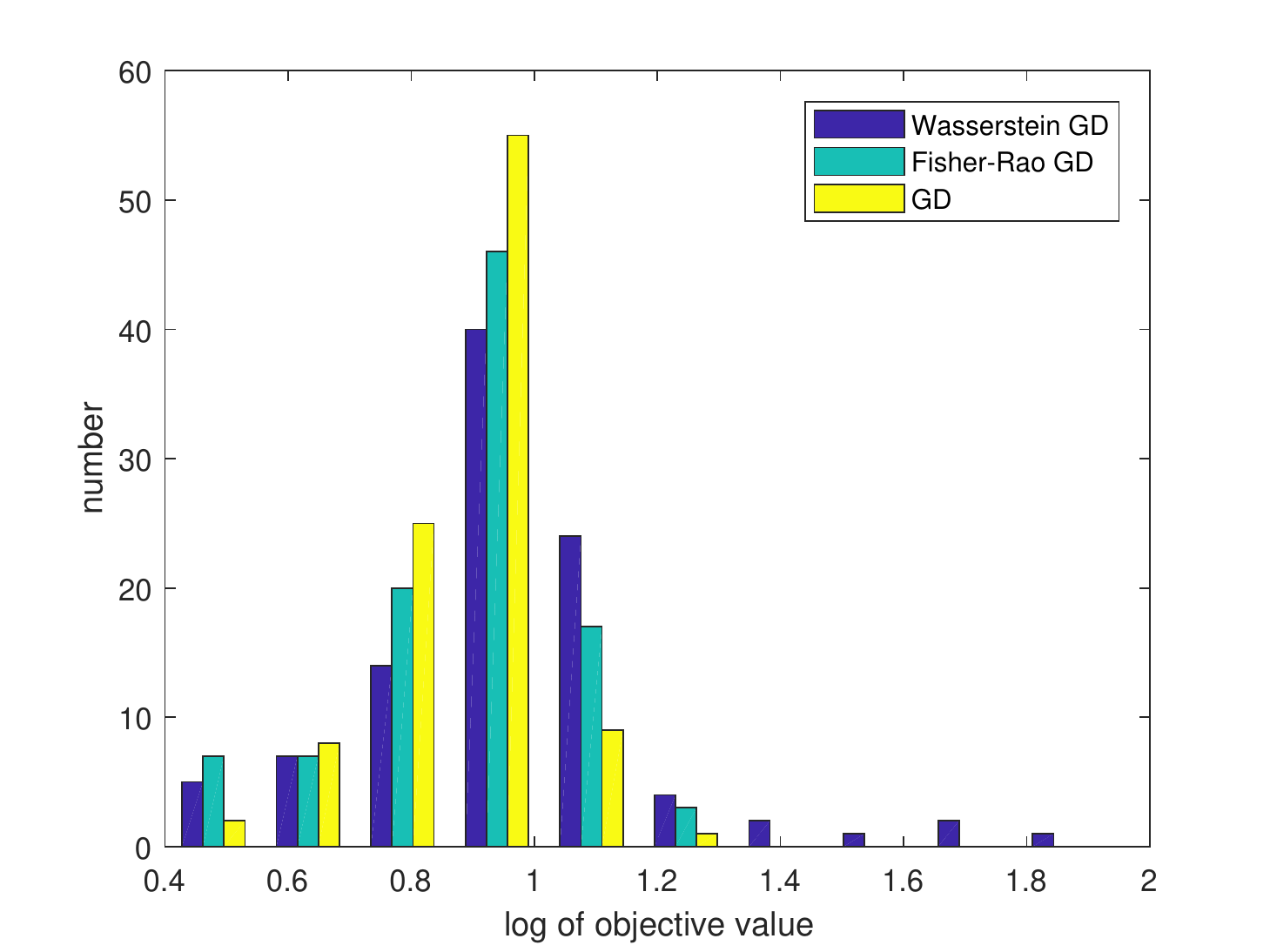}
        \caption{Histogram of obj (Maximum Likelihood Estimate: well-specified case)}
        \label{hist: MLE, well}
    \end{figure}
    From the Tables \ref{table:W2 loss, well}, \ref{table: MLE, well} and Figures \ref{hist:W2 loss, well}, \ref{hist: MLE, well}, we observe that Wasserstein GD behaves well when the Wasserstein metric is used as loss function, while it may perform worse than GD if it is applied to Maximum Likelihood Estimate case. Similarly, Fisher-Rao GD is very efficient in the Maximum Likelihood Estimate case but is less effective than Wasserstein GD when the Wasserstein metric loss function is used. This may be explained by the argument that the natural gradient behaves as an asymptotic Newton's algorithm. Interestingly, in our experiments, the Fisher-Rao GD seems to be more stable and robust (better than GD) when used in the Wasserstein metric loss case, compared to the performance of Wasserstein GD applied to Maximum Likelihood Estimation. 
    \subsection{Misspecified case}
    In the misspecified case, the target distribution is Laplace distribution
    $$\rho^*(x,\mu,b)=\frac{1}{2b}\exp(-\frac{|x-\mu|}{b}),$$ while the parametric family is still Gaussian mixtures. We choose $\mu$ uniformly distributed from $[-10,10]$, and $b$ uniformly from $[1,4]$ to generate the ground truth distributions and draw $200$ samples from it, which will be the empirical data. We conduct 100 runs of experiments using different generated $\rho(x,\theta), \rho^*$. The setting of the experiment is the same as the well-specified case. Table \ref{table:W2 loss, mis} and Figure \ref{hist: MLE, mis} are the results for the Wasserstein metric loss, while Table \ref{table: MLE, mis} and Figure \ref{hist: MLE, mis} are for the Maximum Likelihood Estimate case.
    
    \begin{table}[ht]
        \centering
        \begin{tabular}{llll}
            \hline
            & Wasserstein GD & Fisher-Rao GD & GD       \\ \hline
            obj mean        & 0.2490         & 0.3238       & 0.4986   \\ \hline
            obj std         & 0.2119        & 0.2493        & 0.3860   \\ \hline
            iterations mean & 6.2947         & 10.0211      & 56.3579 \\ \hline
            iterations std  & 1.7127         & 4.4840       & 55.3561 \\ \hline
        \end{tabular}
        \caption{Wasserstein metric loss function: misspecified case}
        \label{table:W2 loss, mis}
    \end{table}
    
    \begin{figure}[ht]
        \centering
        \includegraphics[width=8cm]{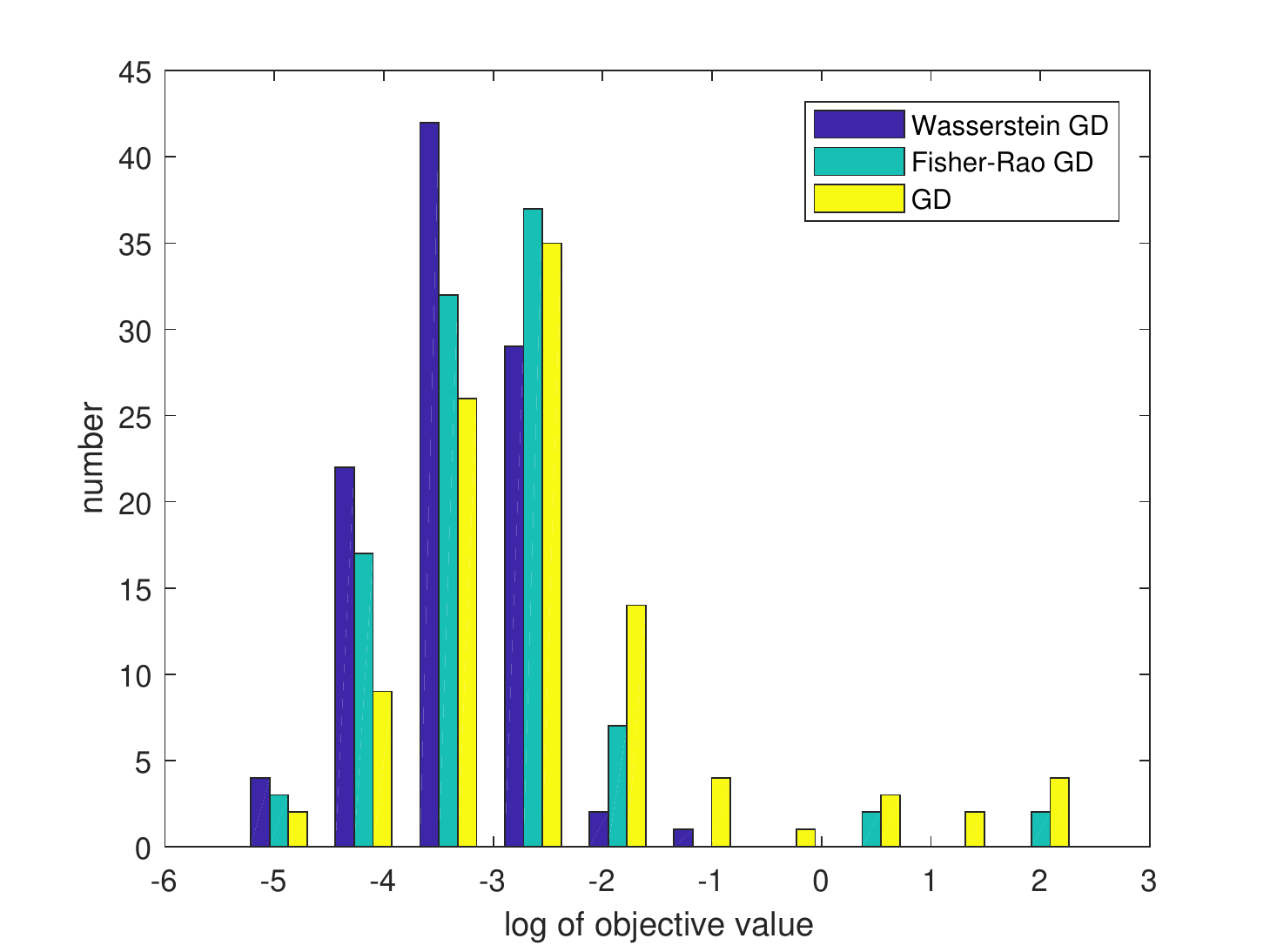}
        \caption{Histogram of obj (Wasserstein metric loss function: misspecified case)}
        \label{hist: W2 loss, mis}
    \end{figure}
    
    \begin{table}[ht]
        \centering
        \begin{tabular}{llll}
            \hline
            & Wasserstein GD & Fisher-Rao GD & GD       \\ \hline
            obj mean        & 2.8450         & 2.6446      & 2.6426   \\ \hline
            obj std         & 0.5207       & 0.3366        & 0.3792   \\ \hline
            iterations mean & 85.5895         & 4.2421       & 16.6211 \\ \hline
            iterations std  & 80.7177         & 1.1552       & 5.6251 \\ \hline
        \end{tabular}
        \caption{Maximum Likelihood Estimate: misspecified case}
        \label{table: MLE, mis}
    \end{table}
    
    \begin{figure}[ht]
        \centering
        \includegraphics[width=8cm]{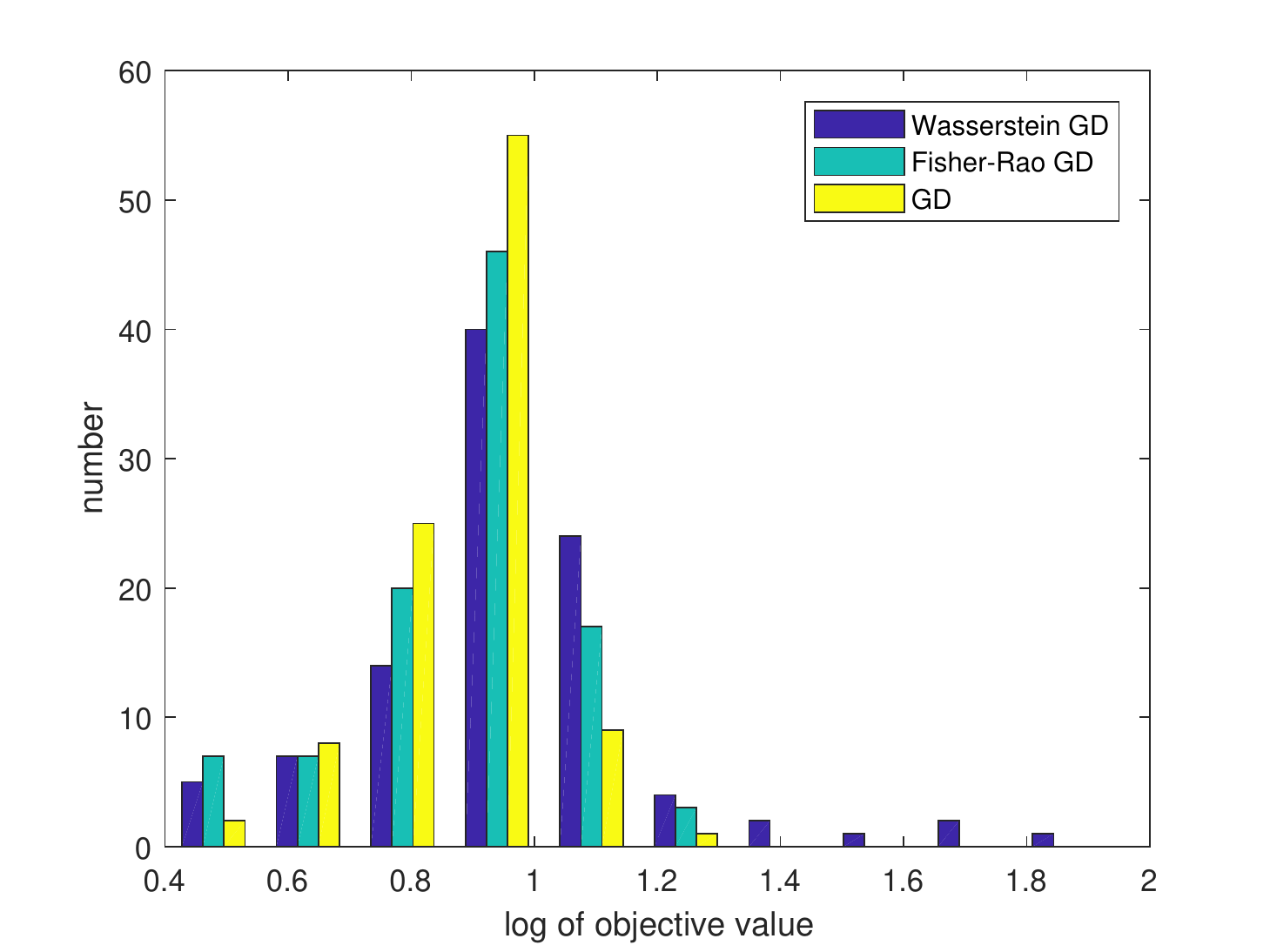}
        \caption{histogram of obj (Maximum Likelihood Estimate: misspecified case)}
        \label{hist: MLE, mis}
    \end{figure}
    
    From these tables and figures, we observe that the results are very similar to the well-specified case. Wasserstein GD is efficient for the Wasserstein metric loss function, while Fisher-Rao GD performs very well for Maximum Likelihood Estimation. These results still hold for the situation that the ground truth distribution $\rho^*$ is not in the parametric family. This implies the robustness of the natural gradient method. Though here we only conduct experiments on Gaussian mixtures and Laplace distribution, we believe our results reveal that when using the natural gradient concept for computations, one needs to choose a suitable geometry to design the natural gradient. This geometry should explore the structures of the objective function and the optimization problem in general. 
    \section{Discussion}
    \label{sec: discussion}
    To summarize, we introduce the Wasserstein statistical manifold for parametric models with continuous sample space. The metric tensor is derived by pulling back the $L^2$-Wasserstein metric tensor in density space to parameter spaces. Given this Riemannian structure, the Wasserstein natural gradient is then proposed. In a one-dimensional sample space, we obtain an explicit formula for this metric tensor, and from it, we show that the Wasserstein natural gradient descent method achieves asymptotically Newton's method for the Wasserstein metric modeled minimizations. Our numerical examples justify these arguments.
    
    One potential future direction is using Theorem \ref{thm: hessian of R} to design various efficient algorithms for solving Wasserstein metric modeled problems. The Wasserstein gradient descent only takes the asymptotic behavior into consideration, and we think a careful investigation of the structure \eqref{eqn:hessian of R} will lead to better non-asymptotic results. Moreover, generalizing \eqref{eqn:hessian of R} to higher dimensions also remains a challenging and interesting issue. We are working on designing an efficient computational method for obtaining $G_W(\theta)$ and hope to report it in subsequent papers.
    
    Analytically, the treatment of the Wasserstein statistical manifold could be generalized. This paper takes an initial step in introducing Wasserstein geometry to parametric models. More analysis on the solution of the elliptic equation and its regularity will be conducted. 
    
    Further, we believe ideas and studies from information geometry could lead to natural extensions in Wasserstein statistical manifold. The Wasserstein distance has shown its effectiveness in illustrating and measuring low dimensional supported densities in high dimensional space, which is often the target of many machine learning problems. We are interested in the geometric properties of the Wasserstein metric in these models, and we will continue to work on it.
    \\
    
    \textbf{Acknowledgments:} This research is partially supported by AFOSR MURI proposal number 18RT0073. The research of Yifan Chen is partly supported by the Tsinghua undergraduate Xuetang Mathematics Program and Caltech graduate Kortchak Scholarship. The authors thank Prof. Shui-Nee Chow for his farseeing viewpoints on the related topics, and we acknowledge many fruitful discussions with Prof. Wilfrid Gangbo and Prof. Wotao Yin. We gratefully thank Prof. Guido Mont\'ufar for many valuable comments regarding the experimental design part about an earlier version of this manuscript.


\begin{thebibliography}{10}
        
        \bibitem{NG}
        S.~Amari.
        \newblock Natural {{Gradient Works Efficiently}} in {{Learning}}.
        \newblock {\em Neural Computation}, 10(2):251--276, 1998.
        
        \bibitem{IG}
        S.~Amari.
        \newblock {\em Information Geometry and Its Applications}.
        \newblock Volume 194. Springer, 2016.
        
        \bibitem{Amari1998Adaptive}
        S.~Amari and A.~Cichocki.
        \newblock Adaptive blind signal processing-neural network approaches.
        \newblock {\em Proceedings of the IEEE}, 86(10):2026--2048, 1998.
        
        \bibitem{LP}
        S.~Amari, R.~Karakida, and M.~Oizumi.
        \newblock Information {{Geometry Connecting Wasserstein Distance}} and
        {{Kullback}}-{{Leibler Divergence}} via the {{Entropy}}-{{Relaxed
                Transportation Problem}}.
        \newblock {\em Information Geometry}, 1(1),13--37, 2018.
        
        \bibitem{Ambrosio2008Gradient}
        L.~Ambrosio, N.~Gigli, and {G.~Savar\'e}.
        \newblock {\em Gradient {{Flows}}: In {{Metric Spaces}} and in the {{Space}} of
            {{Probability Measures}}}.
        \newblock {Birkh{\"a}user Basel}, Basel, 2005.
        
        \bibitem{WGAN}
        M.~Arjovsky, S.~Chintala, and L.~Bottou.
        \newblock Wasserstein {{GAN}}.
        \newblock {\em arXiv:1701.07875 [cs, stat]}, 2017.
        
        \bibitem{IG2}
        N.~Ay, J.~Jost, H.~V. L{\^e}, and L.~J. Schwachh{\"o}fer.
        \newblock {\em Information Geometry}.
        \newblock Ergebnisse der Mathematik und ihrer Grenzgebiete A series of modern surveys in mathematics. Folge, volume 64. {Springer}, Cham, 2017.
        
        \bibitem{BB}
        J.~D. Benamou and Y.~Brenier.
        \newblock A computational fluid mechanics solution to the
        {{Monge}}-{{Kantorovich}} mass transfer problem.
        \newblock {\em Numerische Mathematik}, 84(3):375--393, 2000.
        
        \bibitem{Bernton2017Inference}
        E.~Bernton, P.~E. Jacob, M.~Gerber, and C.~P. Robert.
        \newblock Inference in generative models using the wasserstein distance.
        \newblock {\em arXiv:1701.05146 [math, stat]}, 2017.
        
        \bibitem{BuresWasserstein}
        {R.~Bhatia, T.~Jian, Y.~Lim.
            \newblock On the Bures-Wasserstein distance between positive definite matrices.
            \newblock {\em Expositiones Mathematicae}, 2018.}
        
        \bibitem{C2}
        E.~A. Carlen and W.~Gangbo.
        \newblock Constrained {{Steepest Descent}} in the 2-{{Wasserstein Metric}}.
        \newblock {\em Annals of Mathematics}, 157(3):807--846, 2003.
        
        \bibitem{CarliNingGeorgiou2013_convexa}
        F.~P. Carli, L.~Ning, and T.~T. Georgiou.
        \newblock Convex {{Clustering}} via {{Optimal Mass Transport}}.
        \newblock {\em arXiv:1307.5459 [cs]}, 2013.
        
        \bibitem{Chen2017The}
        J.~Chen, Y.~Chen, H.~Wu, and D.~Yang.
        \newblock The quadratic {{Wasserstein}} metric for earthquake location.
        \newblock {\em Journal of Computational Physics}, 373:188--209, 2018.
        
        \bibitem{ChenGeorgiouTannenbaum2017_optimal}
        Y.~Chen, T.~T. Georgiou, and A.~Tannenbaum.
        \newblock Optimal transport for {{Gaussian}} mixture models.
        \newblock {\em IEEE Access}, 7:6269--6278, 2019.
        
        \bibitem{cencov}
        N.~N. Chentsov.
        \newblock {\em {Statistical decision rules and optimal inference}}.
        \newblock {American Mathematical Society}, Providence, R.I., 1982.
        
        \bibitem{Ligame}
        S.~N. Chow, W.~Li, J.~Lu, and H.~Zhou.
        \newblock Population games and {{Discrete}} optimal transport.
        \newblock {\em Journal of Nonlinear Science}, 29(3):871--896, 2019.
        
        \bibitem{MFG}
        P.~Degond, J.~G. Liu, and C.~Ringhofer.
        \newblock Large-{{Scale Dynamics}} of {{Mean}}-{{Field Games Driven}} by
        {{Local Nash Equilibria}}.
        \newblock {\em Journal of Nonlinear Science}, 24(1):93--115, 2014.
        
        \bibitem{Engquist2014Application}
        B.~Engquist and B.~D. Froese.
        \newblock Application of the {{Wasserstein}} metric to seismic signals.
        \newblock {\em Communications in Mathematical Sciences}, 12(5):979--988, 2014.
        
        \bibitem{Engquist2016Optimal}
        B.~Engquist, B.~D. Froese, and Y.~Yang.
        \newblock Optimal transport for seismic full waveform inversion.
        \newblock {\em Communications in Mathematical Sciences}, 14(8):2309--2330,
        2016.
        
        \bibitem{LWL}
        C.~Frogner, C.~Zhang, H.~Mobahi, M.~Araya-Polo, and T.~Poggio.
        \newblock Learning with a {{Wasserstein Loss}}.
        \newblock In {\em Advances in Neural Information Processing Systems}, pages 2053--2061, 2015.
        
        \bibitem{Lafferty}
        J.~D. Lafferty.
        \newblock The density manifold and configuration space quantization.
        \newblock {\em Transactions of the American Mathematical Society},
        305(2):699--741, 1988.
        
        \bibitem{LiG}
        W.~Li.
        \newblock Geometry of probability simplex via optimal transport.
        \newblock {\em arXiv:1803.06360 [math]}, 2018.
        
        \bibitem{LM}
        W.~Li and G.~Montufar.
        \newblock Natural gradient via optimal transport.
        \newblock {\em Information Geometry}, 1(2): 181-214, 2018.
        
        \bibitem{Lott}
        J.~Lott.
        \newblock Some {{Geometric Calculations}} on {{Wasserstein Space}}.
        \newblock {\em Communications in Mathematical Physics}, 277(2):423--437, 2007.
        
        \bibitem{LV}
        J.~Lott and C.~Villani.
        \newblock Ricci curvature for metric-measure spaces via optimal transport.
        \newblock {\em Annals of Mathematics}, 169(3):903--991, 2009.
        
        \bibitem{WM}
        L.~Malag{\`o}, L.~Montrucchio, and G.~Pistone.
        \newblock Wasserstein {{Riemannian Geometry}} of {{Positive Definite
                Matrices}}.
        \newblock {\em arXiv:1801.09269 [math, stat]}, 2018.
        
        \bibitem{Malag NG exponential family}
        {L.~Malag{\`o} and G.~Pistone.
            \newblock Natural Gradient Flow in the Mixture Geometry of a Discrete Exponential Family.
            \newblock {\em Entropy}, 17(6):4215-4254, 2015.}
        
        \bibitem{Malag NG fitness}
        {L.~Malag{\`o}, M.~Matteucci, and G.~Pistone. 
            \newblock Natural gradient, fitness modelling and model selection: A unifying perspective. 
            \newblock {\em 2013 IEEE Congress on Evolutionary Computation}, Cancun, pages 486-493, 2013.}
        
        \bibitem{Malag RS for NG}
        {L.~Malag{\`o} and M.~Matteucci.
            \newblock Robust Estimation of Natural Gradient in Optimization by Regularized Linear Regression.
            \newblock {\em Geometric Science of Information}, Springer Berlin Heidelberg, pages 861--867, 2013.}
        
        \bibitem{Martens2014New}
        J.~Martens.
        \newblock New insights and perspectives on the natural gradient method.
        \newblock {\em arXiv:1412.1193 [cs, stat]}, 2014.
        
        \bibitem{Marti2016Optimal}
        G.~Marti, S.~Andler, F.~Nielsen, and P.~Donnat.
        \newblock Optimal transport vs. {{Fisher}}-{{Rao}} distance between copulas for
        clustering multivariate time series.
        \newblock {\em 2016 IEEE Statistical Signal Processing Workshop}, pages 1--5, 2016.
        
        \bibitem{MeBrMeOuVi:16}
        {L. M\'etivier, R. Brossier, Q. M\'erigot, E. Oudet and J. Virieux.
            \newblock Measuring the misfit between seismograms using an optimal transport distance: application to full waveform inversion
            \newblock {\em Geophysical Supplements to the Monthly Notices of the Royal Astronomical Society}, 205(1): 345--377, 2016.}
        
        \bibitem{MeBrMeOuVi:16b}
        {L. M\'etivier, R. Brossier, Q. M\'erigot, E. Oudet and J. Virieux.
            \newblock An optimal transport approach for seismic tomography: application to 3D full waveform inversion.
            \newblock {\em Inverse Problems}, 32(11): 115008, 2016.}
        
        \bibitem{IGW}
        K.~Modin.
        \newblock Geometry of {{Matrix Decompositions Seen Through Optimal Transport}}
        and {{Information Geometry}}.
        \newblock {\em Journal of Geometric Mechanics}, 9(3):335--390, 2017.
        
        \bibitem{Boltzman}
        G.~Montavon, K.~R. M{\"u}ller, and M.~Cuturi.
        \newblock Wasserstein {{Training}} of {{Restricted Boltzmann Machines}}.
        \newblock In {\em Advances in {{Neural Information Processing Systems}} 29}, pages 3718--3726, 2016.
        
        \bibitem{Ollivier NG Kalman}
        {Y.~Ollivier. 
            \newblock Online natural gradient as a Kalman filter. 
            \newblock {\em Electronic Journal of Statistics}, 12(2): 2930-2961, 2018.}
        
        \bibitem{Ollivier NG extended}
        {Y.~Ollivier.
            \newblock The Extended Kalman Filter is a Natural Gradient Descent in Trajectory Space. 
            \newblock {\em arXiv:1901.00696}, 2019.}
        
        \bibitem{Ollivier Asym NG}
        {Y.~Ollivier. 
            \newblock True Asymptotic Natural Gradient Optimization. 
            \newblock {\em arXiv:1712.08449}, 2017.}
        
        \bibitem{Olliver IGO}
        {Y.~Ollivier, L.~Arnold, A.~Auger and N.~Hansen.
            \newblock Information-Geometric Optimization Algorithms: A Unifying Picture via Invariance Principles.
            \newblock {\em Journal of Machine Learning Research}, 18(18):1--65, 2017.}
        
        \bibitem{otto2001}
        F.~Otto.
        \newblock The geometry of dissipative evolution equations the porous medium
        equation.
        \newblock {\em Communications in Partial Differential Equations},
        26(1-2):101--174, 2001.
        
        \bibitem{PeyreCuturi2018_computational}
        G.~Peyr{\'e} and M.~Cuturi.
        \newblock Computational {{Optimal Transport}}.
        \newblock {\em arXiv:1803.00567 [stat]}, 2018.
        
        \bibitem{Sanctis2017A}
        A.~De Sanctis and S.~Gattone.
        \newblock A {{Comparison}} between {{Wasserstein Distance}} and a {{Distance Induced}} by {{Fisher}}-{{Rao Metric}} in {{Complex Shapes Clustering}}.
        \newblock {\em Multidisciplinary Digital Publishing Institute Proceedings}, 2(4):163, 2017.
        
        \bibitem{GW}
        A.~Takatsu.
        \newblock Wasserstein geometry of {{Gaussian}} measures.
        \newblock {\em Osaka Journal of Mathematics}, 48(4):1005--1026, 2011.
        
        \bibitem{vil2008}
        C.~Villani.
        \newblock {\em Optimal Transport: Old and New}.
        \newblock Volume 338. Springer Science \& Business Media, 2008.
        
        \bibitem{Wong}
        T.~L. Wong.
        \newblock Logarithmic divergences from optimal transport and {R\'enyi} geometry.
        \newblock {\em Information Geometry}, 1(1): 39-78, 2018.
    \end{thebibliography}
\end{document}